\documentclass[10pt]{amsart}
\usepackage{amsfonts}
\usepackage{array}
\usepackage{tabularx}
\usepackage{arydshln}
\usepackage{amsmath}

\usepackage{amssymb}
\usepackage{physics}
\usepackage{graphicx}
\usepackage{caption}
\usepackage{subcaption}
\usepackage{multirow}
\usepackage{footmisc}
\providecommand{\U}[1]{\protect\rule{.1in}{.1in}}

\usepackage{xcolor}
\usepackage{latexsym}
\usepackage{amsmath}
\usepackage{amssymb}
\usepackage{mathrsfs}
\usepackage{graphicx}
\usepackage{color}
\usepackage{pgfpages}
\usepackage{ifthen}
\usepackage{leftidx,tensor}
\usepackage[T1]{fontenc}
\usepackage[latin1]{inputenc}
\usepackage{mathtools}
\usepackage{comment}
\usepackage{dsfont}
\usepackage[nocompress]{cite}
\usepackage[shortlabels]{enumitem}
\usepackage{aliascnt}
\usepackage[bookmarks=true,pdfstartview=FitH, pdfborder={0 0 0}, colorlinks=true,citecolor=red, linkcolor=blue]{hyperref}
\usepackage{nicefrac}
\newtheorem{thm}{Theorem}[section]
\newaliascnt{cor}{thm}
\newaliascnt{prop}{thm}
\newaliascnt{lem}{thm}
\newtheorem{cor}[cor]{Corollary}
\newtheorem{prop}[prop]{Proposition}
\newtheorem{lem}[lem]{Lemma}
\aliascntresetthe{cor}
\aliascntresetthe{prop}
\aliascntresetthe{lem}

\newaliascnt{defn}{thm}
\newaliascnt{asu}{thm}
\newaliascnt{con}{thm}

\aliascntresetthe{defn}
\aliascntresetthe{asu}
\aliascntresetthe{con}

\newcounter{stp}
\newcounter{stpi}
\newcounter{stpci}
\newcounter{stpiii}

\newaliascnt{rem}{thm}
\newaliascnt{exa}{thm}
\newaliascnt{masu}{thm}
\newaliascnt{nota}{thm}
\newaliascnt{sett}{thm}
\newtheorem{rem}[rem]{Remark}

\aliascntresetthe{rem}
\aliascntresetthe{exa}
\aliascntresetthe{masu}
\aliascntresetthe{nota}
\aliascntresetthe{sett}
\numberwithin{equation}{section}

\setlist[enumerate]{font = \normalfont}

\newcommand {\N}	{\mathbb{N}}

\newcommand {\R}	{\mathbb{R}}

\newcommand {\E}	{\mathbb{E}}

\newcommand {\T}	{\mathbb{T}}

\renewcommand{\d}{\, \mathrm{d}}

\renewcommand{\i}{\mathrm{i}}

\newcommand{\D}{\mathrm{D}}

\renewcommand{\H}{\mathrm{H}}
	\def\rot{{\rm curl} \,}

	\newcommand{\eps}{\varepsilon}
	\renewcommand{\phi}{\varphi}
	
	\renewcommand{\div}{\mathrm{div} \, }
	\newcommand{\rC}{\mathrm{C}}
	\newcommand{\rL}{\mathrm{L}}
	\newcommand{\rW}{\mathrm{W}}
	\newcommand{\rH}{\H}

	\renewcommand{\P}{\mathcal{P}}

       \newcommand{\cV}{\mathcal{V}}
         \newcommand{\cH}{\mathcal{H}}
     \newcommand{\rI}{\mathrm{I}}

\title[Continuous data assimilation for semilinear parabolic equations with noisy observations]{Continuous Data Assimilation for Semilinear Parabolic Equations with Multiplicative Observation Noise}

\author{Jochen Br\"ocker}
\address{University of Reading\\
School of Mathematical, Physical, and Computational Sciences\\
and Centre for the Mathematics of Planet Earth\\
	Whiteknights, PO Box 220\\
	Reading RG6 6AX\\
	United Kingdom}
\email{j.broecker@reading.ac.uk}

\author{Gianmarco Del Sarto}
\address{Technische Universit\"{a}t Darmstadt\\
Fachbereich Mathematik\\
	Schlossgartenstr.\ 7\\
	64289 Darmstadt\\
	Germany}
\email{delsarto@mathematik.tu-darmstadt.de}

\author{Matthias Hieber}
\address{Technische Universit\"at Darmstadt\\
	Fachbereich Mathematik\\
	Schlossgartenstr.\ 7\\
	64289 Darmstadt\\
	Germany}
\email{hieber@mathematik.tu-darmstadt.de}

\author{Filippo Palma}
\address{Universit\`a degli Studi della Campania L. Vanvitelli\\
	Dipartimento di Matematica e Fisica\\
	Via Vivaldi 43\\
	81100 Caserta\\
	Italy}
\email{filippo.palma@unicampania.it}

\author{Tarek Z\"{o}chling}
\address{Technische Universit\"at Darmstadt\\
	Fachbereich Mathematik\\
	Schlossgartenstr.\ 7\\
	64289 Darmstadt\\
	Germany}
\email{zoechling@mathematik.tu-darmstadt.de}

\begin{document}
\keywords{Continuous data assimilation; noisy observations; multiplicative noise; semilinear parabolic equations; stochastic partial differential equations}

\subjclass[2020]{Primary 35K58; Secondary 60H15, 93C20, 35B40, 35B41, 37L30}
\begin{abstract}
The problem of continuous data assimilation for semilinear parabolic equations based on partial observations corrupted by noise is investigated. The noise is allowed to be \emph{multiplicative}, with additive noise arising as a special case. In a general Gelfand triple framework, an abstract theory for the nudging equation is developed that covers both weak and strong formulations. Mean square convergence of the assimilation error is proved under suitable assumptions, and, under additional integrability conditions on the noise, a uniform almost sure convergence result is established. Finally, the framework is applied to several PDE models, including the 2D Navier-Stokes, 2D magnetohydrodynamics, 2D quasi-geostrophic, and 1D Allen-Cahn equations.
\end{abstract}
\maketitle

%\tableofcontents

\section{Introduction}
\noindent 
Continuous data assimilation aims at reconstructing the state of an evolution equation from incomplete observations by combining the available measurements with the model dynamics through a suitable feedback mechanism \cite{Kalnay2003AtmosphericDA,ReichCotter2015PFDA,Evensen}. It is a central ingredient in modern forecasting, especially in numerical weather prediction, where data assimilation is a major component of the forecasting system and one of its most computationally demanding parts; see, for instance, \cite{Rabier2003,Lean2021}. In continuous time, this is often achieved by \emph{nudging}: the reconstructed trajectory is evolved by an identical copy of the underlying evolution equation, but continuously corrected by feeding into the equation the discrepancy between the observed coarse-scale data and the corresponding coarse-scale prediction produced by the reconstructed trajectory \cite{AOT2014,Stuart2015}.

Our mathematical setup is broadly as follows, see \autoref{sec: main results} for details. Let $(\cV,\cH,\cV^\ast)$ be a Gelfand triple, let $A\colon \cV\to \cV^\ast$ be a bounded linear operator, and let $u_0\in \cH$. We consider, for $T>0$, the semilinear parabolic evolution equation
\begin{equation}\label{eq:model}
    \left\{
    \begin{aligned}
        u'+Au&=F(u) , \quad t \in (0,T),\\
        u(0)&=u_0 \in \cH.
    \end{aligned}
    \right.
\end{equation}
Here $u$ is regarded as the unknown reference, or true, trajectory. Since the initial condition $u_0$ is not available to the observer, the full trajectory $u$ is itself unknown. The problem of data assimilation is then to use only partial observations of this trajectory to construct a second trajectory that synchronises with it.

We assume that only coarse-scale observations of $u$ are available through an interpolation/observation operator $\rI_\delta$, where $\delta>0$ denotes the observation scale. The operator $\rI_\delta$ will approximate the identity as $\delta \to 0$, in a sense to be made precise later. A central feature of the present paper is that these observations are noisy, and the noise intensity may depend on the reference trajectory itself. More precisely, we assume that the observation process $(y_t)_{t\ge 0}$ satisfies 
\begin{equation}\label{eq:obs-process-u}
    \d y_t=\rI_\delta u(t) \d t + G_{\delta}(u(t)) \d W_t^Q .
\end{equation}
Here, $(W_t^Q)_t$ denotes a $Q$-Wiener process on $\cH$. 
In other words, we consider \emph{multiplicative noise} in the observation process, with the additive-noise case recovered when the noise coefficient $G_{\delta}$ is independent of the solution.
This setting of multiplicative noise is physically relevant, since realistic observations are affected not only by instrument and representativeness errors, but often also by state-dependent or flow-dependent uncertainties, see, e.g., \cite{Geer2011,Carrassi2018,Bishop2019}.
Furthermore, we assume the noise coefficient $G_{\delta}$ to be dependent on the observation scale $\delta$.
Following the discussion in~\cite{Bessaih2015}, our motivation for this is that in typical applications, the ``raw'' observations delivered by the measurement apparatus are of the form
\begin{equation}\label{eq:obs-raw}
    \d o_t = M_\delta u(t) \d t + \sigma_{\delta}(u(t)) \d B_t,
\end{equation}
where $M_\delta : \cH \to \R^N$ is a finite rank operator, $B$ is a standard $N$--dimensional Wiener process, and $\sigma_{\delta}: \cH \to \R^{N \times N}$ for each $\delta > 0$, with $N$ typically increasing as $\delta$ decreases.
The interpolation of those ``raw'' observations is what gives rise to $\{y_t\}_{t \geq 0}$ in~\eqref{eq:obs-process-u} by applying an interpolation operator $L_\delta : \R^N \to \cH$ to~\eqref{eq:obs-raw}.
This gives $\rI_\delta = L_\delta M_\delta$ and $G_{\delta} = L_\delta \sigma_{\delta}$.
We also stress that even if $\sigma_{\delta}$ does not depend on $u$, it should be generally permitted to depend on $\delta$.
Typically, the standard $N$--dimensional Wiener process in~\eqref{eq:obs-raw} needs to be rescaled with $\delta$ as otherwise a trivial limit $G_{\delta} \to 0$ for $\delta \to 0$ might occur.
Starting from these observations, we introduce a reconstructed trajectory $v$, initialized from an arbitrary datum $v_0\in\cH$, and evolve it according to the same model dynamics as $u$ (i.e.\ equation~\eqref{eq:model}), but with an additional correction term. This is the \emph{nudging term}, namely the feedback of the error
$$
\rI_\delta u(t)-\rI_\delta v_t,
$$
which measures the discrepancy between the observed and predicted coarse-scale data. The resulting stochastic data-assimilation system is 
\begin{equation}\label{eq:model-data-stoch-u}\tag{\textcolor{blue}{DA$_{\mathrm{sto}}$}}
    \left\{
    \begin{aligned}
        \d v_t + Av_t\d t
        &= F(v_t)\d t - \mu\bigl(\rI_\delta v_t-\rI_\delta u(t)\bigr)\d t + \mu\, G_{\delta}(u(t))\d W_t^Q,
        \qquad t\in (0,T),\\
        v(0)&=v_0\in \cH.
    \end{aligned}
    \right.
\end{equation}
where $\mu>0$ is the nudging parameter. Larger values of $\mu$ correspond to a stronger feedback, whereas $\delta$ determines the amount of information carried by the observations.

Our aim is to investigate the data-assimilation error
$$
w:=u-v.
$$
Our first group of main results gives quantitative mean-square estimates for the assimilation error. Under suitable structural assumptions on the drift, the observation operator, and the noise coefficient, we prove that the error decays exponentially fast in mean square up to a stochastic residual term, provided $\mu$ is large enough to stabilise the data-assimilation system~\eqref{eq:model-data-stoch-u}, while $\mu\delta^2\lesssim 1$, where $\delta$ measures the size of $\rI-\rI_\delta$ as an operator from $\cH$ to $\cV^\ast$. Thus, in the presence of non-vanishing observation noise, synchronisation holds only up to a noise-dependent remainder. If the noise intensity is uniformly bounded along the reference trajectory, this yields an explicit \emph{noise floor}, namely an upper bound on the asymptotic mean-square error. We also obtain the corresponding estimate in the $\cV^\ast$-norm. Finally, when the deterministic dynamics admits a compact global attractor and the noise coefficient is continuous near it, the asymptotic error can be controlled only in terms of the values of the noise on the attractor, rather than along the full reference trajectory.

Our final main result concerns almost sure synchronisation. Under additional integrability assumptions on the stochastic forcing along the reference trajectory, we prove that the residual stochastic contribution becomes asymptotically negligible and that the convergence can be upgraded to $\mathbb{P}$-almost sure convergence, uniformly on the tail. In particular,
$$
    \sup_{t\geq N} \| u(t) - v_t\|_{\cH}\to 0 \qquad \mathbb{P} \text{-a.s. as } N \to \infty .
$$
This attractor-based viewpoint is especially useful for long-time dynamics: if the noise vanishes on the attractor and is locally Lipschitz near it, then the additional integrability condition is naturally satisfied under a suitable rate of attraction. To the best of our knowledge, this uniform almost sure synchronisation result is essentially new in the literature on stochastic continuous data assimilation with multiplicative noise. The only related result we are aware of is the recent work \cite{Bessaih2025}, where an almost sure convergence result is obtained for the two-dimensional stochastic Navier-Stokes equations (NSEs), with multiplicative noise in the dynamics rather than in the observations.

Another feature of the paper is that the analysis is carried out in a flexible abstract setting, which allows us to treat both weak and strong formulations within the same framework. In the weak setting, we apply the theory to the two-dimensional NSEs, the two-dimensional magnetohydrodynamics equations, the two-dimensional quasi-geostrophic equations, and the one-dimensional Allen--Cahn equation. We also revisit the two-dimensional NSEs and the one-dimensional Allen--Cahn equation in a stronger functional setting, obtaining convergence results in stronger norms.

Let us briefly place our results in the existing literature. Our abstract framework is closely connected with our recent work on deterministic continuous data assimilation for semilinear parabolic equations \cite{CDA}. More broadly, deterministic data assimilation has been studied for a variety of models, including the 2D B\'enard convection problem \cite{FarhatJollyTiti2015Benard}, the 3D Navier-Stokes-$\alpha$ model \cite{Titialpha}, the 3D Navier-Stokes equations \cite{MR4344886}, reaction-diffusion equations \cite{Larios}, and the two-dimensional Cahn-Hilliard-Navier-Stokes system \cite{MR4409797}. In the stochastic or noisy-data setting, classical and more recent contributions, both abstract and computational, include \cite{Bessaih2015,Blomker2013,Hammoud2022,Bessaih2025,WangSIAM}. In particular, \cite{Bessaih2015} treats the two-dimensional NSEs with additive noise in the observations, while \cite{Bessaih2025} studies continuous data assimilation for the two-dimensional stochastic NSEs with multiplicative noise in the dynamics, but no noise in the observations. Compared with these works, the present paper treats multiplicative noise in the observations and develops an abstract variational theory for a broad class of semilinear parabolic equations, yielding synchronisation in mean square and, under suitable assumptions, also almost surely.

The paper is organised as follows. In \autoref{sec: main results} we introduce the abstract framework and state the main results. \autoref{sec: examples} illustrates the theory on several concrete PDE models in both weak and strong settings. In \autoref{sec: preparations} we establish the preparatory well-posedness results for both the reference and the assimilated systems. \autoref{sec: proofs main results} is devoted to the proofs of the abstract convergence theorems. \autoref{sec:appendix} contains an auxiliary local well-posedness result for non-autonomous semilinear equations, which is used in the proof of the stochastic data-assimilation well-posedness result.

\section{Assumptions and main results} \label{sec: main results}
\noindent 
In this section we state the structural assumptions on the deterministic equation \eqref{eq:model}, on the observation operator $\rI_\delta$, and on the noise coefficient $G_{\delta}$, and then present our main convergence results for the data-assimilation error.

A \emph{Gelfand triple} is a triple of real Hilbert spaces
$$
\cV \hookrightarrow \cH \hookrightarrow   \cV^\ast,
$$
with dense and continuous embeddings, where $\cH$ is identified with its dual and
$$
\langle u, v\rangle_{\cV^\ast,\cV}=(u,v)_\cH,
\qquad u\in \cH, v\in \cV.
$$
For a compatible pair of Banach spaces $(X_0,X_1)$, we denote by $(X_0,X_1)_{\theta,p}$ and $[X_0,X_1]_\theta$ the real and complex interpolation spaces, respectively. \par Let $\D\subseteq\R^d$, with $d\in\N$, be a sufficiently regular domain. For $m\in\mathbb N$ and $q \in (1,\infty)$, we denote by $\rL^q(\D)$ and $\rH^{m,q}(\D)=\rW^{m,q}(\D)$ the Lebesgue and Sobolev spaces, and by $\|\cdot\|_q$ and $\|\cdot\|_{m,q}$ their norms. As is usual in the literature, in the case $q=2$ we set $\rW^{m,2}(\D)=\rH^m(\D)$ and we denote by $\rH^m_0(\D)$ the subset of $\rH^m(\D)$ of functions with zero trace. We recall that, see \cite{Sol}, the space $\rL^q(\D)$, $q\in(1,\infty)$, admits the decomposition
\[
\rL^q(\D)=\rL^q_\sigma(\D)\oplus G^q(\D)\,,
\]
where $\rL^q_\sigma(\D)$ is the subset of functions in $\rL^q(\D)$ that are weakly divergence-free and that have zero generalized trace, while $G^q(\D)$ is the subset of functions $v$ in $\rL^q(\D)$ such that $v=\nabla h$, $h\in \rW^{1,q}_{loc}(\D)$. We denote by $\P_q$ the Helmholtz projection
\[
\P_q:\rL^q(\D)\to \rL^q_\sigma(\D)\,.
\]
In the case $q=2$, we set $\P_2=\P$. For Banach spaces $X$ and $Y$, we denote by $\mathcal{L}(X,Y)$ the space of bounded linear operators from $X$ to $Y$, and with $X'$ the dual space of $X$. For more details on interpolation and function spaces we refer, for instance, to \cite{AdamsFournier,Lunardi2018}.

We work on a complete filtered probability space $\left(\Omega, \mathcal{F}, \left( \mathcal{F}_t \right)_t , \mathbb{P}\right)$. A stochastic process $\Phi$, taking values in a measurable space, is adapted if $\Phi_t$ is $\mathcal{F}_t$-measurable for any $t \geq 0$. It is progressively measurable if the map $(s, \omega) \mapsto \Phi_s(\omega)$ is measurable on $([0,t] \times \Omega, \mathcal{B}([0,t]) \otimes \mathcal{F}_t  )$ for every $t \geq 0$, with $\mathcal{B}([0,t])$ being the Borel $\sigma$-algebra on $[0,t].$ Given a real, separable, Hilbert space $\cH$, we denote by $(W_t^Q)_t$ a $Q$-Wiener process on $\cH$, where $Q$ is a non-negative, self-adjoint, trace-class operator on $\cH$. If $(e_k)_k \subset \cH$ denotes an orthonormal basis of $\cH$ made of eigenvectors for $Q$, i.e. $Q e_k = \lambda_k^2 e_k$, then
\[
W_t^Q = \sum_{k = 0}^\infty  \lambda_k e_k \beta_k(t), \qquad t \geq 0,
\]
where $(\beta_k(t))_k$ are independent Brownian motions on $(\Omega, \mathcal{F},\left( \mathcal{F}_t\right)_t, \mathbb{P})$, and $(\lambda_k)_k$ are the non-negative square roots of the eigenvalues. We also denote by $\rL_2(U,H)$ the space of Hilbert-Schmidt operators between Hilbert spaces $U$ and $H$. For further details, we refer to \cite{DPZ}.

\subsection{The setting}
\label{subsec: assumptions}
Let $(\cV,\cH,\cV^\ast)$ be a \emph{Gelfand triple} of real Hilbert spaces such that the pairing between $\cV$ and $\cV^\ast$ satisfies
$$
\langle u, v\rangle_{\cV^\ast,\cV}=(u,v)_\cH,
\qquad u\in \cH, v\in \cV.
$$
We further assume that for the real interpolation space it holds
\[
(\cV^\ast, \cV)_{\frac{1}{2},2}= \cH.
\]
For $\beta \in (\frac12,1)$, we set
\[
\cV_\beta=[\cV^*,\cV]_\beta\,.
\]
Our first group of assumptions is purely deterministic. 
\begin{description}
    \item[(A1)]
    $A \in \mathcal{L} (\cV, \cV^\ast)$ is coercive, i.e., for all $u\in \cV$ it holds that
    \begin{equation*}
        \langle Au,u \rangle_{\cV^\ast,\cV} \geq \alpha \|u\|^2_\cV
        \qquad \text{for some } \alpha>0.
    \end{equation*}

    \item[(A2)]
The nonlinearity can be written as
\[
    F=\sum_{j=1}^m F_j ,
\]
where, for each $j=1,\dots,m$, there exist
$\beta_j\in(\frac12,1)$, $\rho_j\ge0$, and $C_j>0$ such that
\begin{equation}\label{eq:A2_Fj_bound}
    \|F_j(u)-F_j(v)\|_{\cV^\ast}
    \le
    C_j
    \bigl(1+\|u\|_{\cV_{\beta_j}}^{\rho_j}
            +\|v\|_{\cV_{\beta_j}}^{\rho_j}\bigr)
    \|u-v\|_{\cV_{\beta_j}},
    \qquad u,v\in\cV .
\end{equation}
Moreover,
\begin{equation}\label{eq:A3_coefficients_bound}
    (2\beta_j-1)(\rho_j+1)\le 1,
    \qquad j=1,\dots,m .
\end{equation}
\end{description}
Assumption $\mathbf{(A1)}$ gives the coercivity of the linear part, while $\mathbf{(A2)}$ controls the nonlinear term in a critical interpolation scale. Together, these two assumptions guarantee the local well-posedness of \eqref{eq:model}.
\begin{rem}\label{rem: bilinear}{\rm
Several models in fluid mechanics, such as the Navier-Stokes equations, have a bilinear nonlinearity $
F(u)=B(u,u),$ with a bounded $ B:\cV_\beta\times\cV_\beta\to\cV^\ast $. Then
$
\|F(u)\|_{\cV^\ast}\le C\|u\|_{\cV_\beta}^2,
$
and ${\bf (A2)}$ holds with $m=1$, $\rho_1=1$ and $\beta_1 = \beta \in ( \frac{1}{2}, \frac{3}{4}]$.}
\end{rem}
\begin{rem}
\label{rem:A-stoch-max-reg} \noindent
    \begin{enumerate}
        \item[(i)] Note that, by general theory, assumption \textup{(A1)} implies that $-A$ generates an analytic semigroup on $\cV^\ast$. 
        \item[(ii)] Since $(\cV^\ast,\cV)_{\frac{1}{2},2}=\cH$, the operator $A$ enjoys $\rL^2$-stochastic maximal regularity, see \cite[Theorem 3.13]{Agresti_Veraar_survey}.
        \end{enumerate}
\end{rem}
\noindent
 To exclude finite-time blow-up and to obtain well-posedness of the data-assimilation system, we impose the following structural assumption on $F$.

\begin{description}
\item[(A3)]
For each $j=1,\dots,m$, one of the following alternatives holds.

\begin{itemize}
    \item[(i)]
    If $\frac34\le\beta_j<1$, then there exist
    $\eps_j>0$ and $C_j^{(0)},C_j^{(1)}\ge0$ such that
    \begin{equation}\label{eq:A3diag}
        \langle F_j(x),x\rangle_{\cV^\ast,\cV}
        \le
        \eps_j\|x\|_{\cV}^2
        + C_j^{(1)}\|x\|_{\cH}^2
        + C_j^{(0)},
        \qquad x\in\cV .
    \end{equation}

    \item[(ii)]
    If $\frac12<\beta_j<\frac34$, then there exist
    $\eps_j>0$, $C_j^{(0)},C_j^{(1)}\ge0$, and a measurable map
    $\Psi_j:\cV\to[0,\infty)$ such that
    \begin{equation}\label{eq:A3shift}
        \langle F_j(x+y),x\rangle_{\cV^\ast,\cV}
        \le
        \eps_j\|x\|_{\cV}^2
        + C_j^{(0)}\bigl(1+\Psi_j(y)\bigr)\|x\|_{\cH}^2
        + C_j^{(1)}\Psi_j(y),
        \qquad x,y\in\cV .
    \end{equation}
    Moreover, for every $t>0$ and every
    $z\in\rC([0,t];\cH)\cap\rL^2(0,t;\cV)$, one has
    $\Psi_j(z(\cdot))\in\rL^1(0,t)$.
\end{itemize}
Furthermore, the constants $\eps_j$ satisfy $
    \sum_{j=1}^m \eps_j < \frac{\alpha}{4}.$
\end{description}
\noindent
We now fix $u_0\in\cH$ and denote by $u$ the corresponding solution of
\eqref{eq:model}, whose existence and uniqueness in
$$\rL^2(0,T;\cV) \cap \rH^1(0,T;\cV^\ast) \cap\mathrm{BUC}([0,T];\cH)
$$
follow
from \textup{\bf(A1)}--\textup{\bf(A3)}, see \autoref{sec: preparations}. We
call $u$ the \emph{reference trajectory}. We impose the following natural integrability condition on the noise coefficient $G_{\delta}$.
\begin{description}
    \item[(N1)]
    Assume $\cH$ separable. Let $\cH_0 := Q^{1/2} \cH $ and $\rL_2^0 := \rL_2(\cH_0, \cH)$. 
    There exists $\delta_0>0$ such that, for every $0<\delta\leq\delta_0$
and every reference trajectory $u$, 
    \begin{equation}\label{eq:N1}
        \int_0^T \|G_{\delta}(u(t))\|_{\mathrm{L}_2^0}^2\d t <\infty, \qquad T >0.
    \end{equation}
\end{description}
Note that we do {\em not} impose bounds on $G_{\delta}$ uniform in $\delta$.

The next assumption is an additional
one-sided stability condition on the increments of $F$ along $u$, used to
control the nonlinear term in the error equation. Here $\alpha$ is the coercivity constant from \textup{\bf(A1)}.

\begin{description}
    \item[(A4)]
    There exist $\eps_1 \in (0,\frac{\alpha}{4})$, constants $M_0,M_1\ge 0$, and a non-negative measurable function
    \[
    \kappa_u \colon [0,\infty)\to [0,\infty)
    \]
    such that
    \begin{equation}\label{eq:A4}
        \langle F(u(t)) - F(u(t)-\xi),\xi\rangle_{\cV^\ast,\cV}
        \leq \eps_1 \|\xi\|_{\cV}^2 + \kappa_u(t)\|\xi\|_{\cH}^2,
        \qquad \forall\, \xi\in \cV,
    \end{equation}
    for a.e.\ $t\ge 0$, and
    \begin{equation}\label{eq:A4avg}
        \int_s^t \kappa_u(r)\d r \le M_0 (t-s) + M_1,
        \qquad \forall\, 0\le s\le t.
    \end{equation}
\end{description}
\noindent
Lastly, the observation operator $\rI_\delta$ is assumed to satisfy
\begin{equation}\label{eq:Idelta-bound}
    \langle f-\rI_\delta f,g\rangle_{\cV^\ast,\cV}
    \leq C_{\rI} \delta \|f\|_{\cH}\|g\|_{\cV}
    \quad \text{ for all $f \in \cH$, $g\in \cV$, and $0< \delta \leq \delta_0$},
\end{equation}
for a constant $C_\rI>0$.
\begin{rem}\label{rem:combined-R}
{\rm
A natural structural condition which simultaneously implies the shifted well-posedness assumption \textup{\bf(A3)}-(ii) and the convergence assumption \textup{\bf(A4)} is the following.

\begin{description}
    \item[(R)]
    There exist $\eps_R\in (0,\frac{\alpha}{4})$, a constant $C_R\ge 0$, and a measurable map
    $
    \Psi \colon \cV \to [0,\infty)$ such that
    \begin{equation*}%\label{eq:R}
        \langle F(x+y)-F(x),y\rangle_{\cV^\ast,\cV}
        \le
        \eps_R \|y\|_{\cV}^2
        + C_R\bigl(1+\Psi(x)\bigr)\|y\|_{\cH}^2,
        \qquad \forall\, x,y\in \cV.
    \end{equation*}
    Moreover, for every $T>0$ and every $z \in \rC ([0,T]; \cH) \cap \rL^2(0,T; \cV)$, one has $ \Psi (z(\cdot)) \in \rL^1(0,T).$  Finally, along the reference trajectory $u$, there exist constants $M_0,M_1 \geq 0$ such that 
    \begin{equation}
        \int_s^t \Psi( u(r))\d r \le M_0(t-s)+M_1,
        \label{eq:Ravg-u}
    \end{equation}
    for any $0 \leq s \leq t.$
\end{description}
}
\end{rem}
\smallskip
\noindent
Condition \textup{\bf(R)} gives a practical way of verifying the two
one-sided estimates used in the analysis. First, \textup{\bf(R)} yields \textup{\bf(A4)}. Taking
$x=u(t)$ and $y=-\xi$ in \textup{\bf(R)} gives
\[
    \langle F(u(t))-F(u(t)-\xi),\xi\rangle_{\cV^\ast,\cV}
    \le
    \eps_R\|\xi\|_{\cV}^2
    + C_R\bigl(1+\Psi(u(t))\bigr)\|\xi\|_{\cH}^2 .
\]
Hence \textup{\bf(A4)} holds with
$\kappa_u(t)=C_R(1+\Psi(u(t)))$, and \eqref{eq:A4avg} follows from
\eqref{eq:Ravg-u}. Second, \textup{\bf(R)} also gives the shifted estimate needed in
\textup{\bf(A3)}-(ii). Writing
\[
    \langle F(x+y),x\rangle_{\cV^\ast,\cV}
    =
    \langle F(x+y)-F(y),x\rangle_{\cV^\ast,\cV}
    +
    \langle F(y),x\rangle_{\cV^\ast,\cV},
\]
the first term is controlled by \textup{\bf(R)}
\[
    \langle F(x+y)-F(y),x\rangle_{\cV^\ast,\cV}
    \le
    \eps_R\|x\|_{\cV}^2
    + C_R\bigl(1+\Psi(y)\bigr)\|x\|_{\cH}^2 .
\]
The second term is controlled by \textup{\bf(A2)} and Young's inequality
\[
    |\langle F(y),x\rangle_{\cV^\ast,\cV}|
    \le
    \varepsilon \|x\|_{\cV}^2
    +
    C_\varepsilon
    \Bigl(
        1+\sum_{j=1}^m
        \|y\|_{\cV_{\beta_j}}^{2(\rho_j+1)}
    \Bigr).
\]
The interpolation restriction $
    (2\beta_j-1)(\rho_j+1)\le1$ ensures that the last quantity is integrable whenever
$y\in\rC([0,T];\cH)\cap\rL^2(0,T;\cV)$. Therefore, after replacing $\Psi$ by
\[
    \widetilde\Psi(y)
    :=
    \Psi(y)
    +
    1
    +
    \sum_{j=1}^m
    \|y\|_{\cV_{\beta_j}}^{2(\rho_j+1)},
\]
condition \textup{\bf(R)} implies the shifted one-sided estimate required in
\textup{\bf(A3)}-(ii), as well as the stability assumption \textup{\bf(A4)}. In the componentwise formulation of \textup{\bf(A3)}, the same argument can be
applied to each $F_j$, with constants $\eps_{R,j}>0$ chosen so that $ \sum_{j=1}^m \eps_{R,j}<\frac{\alpha}{4}.$

\subsection{Main results}
\noindent
We now introduce the assimilation error
\[
w:=u-v,
\]
which measures the discrepancy between the reference solution $u$ and the nudged stochastic approximation $v$. The error equation for $w$ is obtained by subtracting \eqref{eq:model-data-stoch-u} from \eqref{eq:model}, 
\begin{equation}\label{eq:model-difference-stoch}
    \left\{
    \begin{aligned}
        \d w_t + Aw_t\d t
        &= \bigl(F(u(t))-F(v_t)\bigr)\d t
        - \mu \, \rI_\delta w_t\d t
        - \mu G_{\delta}(u(t))\d W_t^Q,
        \qquad t\in (0,T),\\
        w(0)&=w_0:=u_0-v_0.
    \end{aligned}
    \right.
\end{equation}
\noindent
\noindent
The next theorem gives the basic mean-square estimate for the assimilation
error. For sufficiently large $\mu$ and sufficiently small $\delta$, with
$\mu\delta^2$ bounded, the deterministic part of the error is exponentially damped, while the remaining term is driven by the observational noise.
\begin{thm}[Mean-square convergence up to the stochastic forcing]
\label{thm:stoch-H}
Let $u_0,v_0\in \cH$, and suppose that $\bf(A1)$--$\bf(A4)$ and $\bf(N1)$ are satisfied. Assume moreover that $\rI_\delta$ satisfies \eqref{eq:Idelta-bound}, and let
$\delta_0>0$ be chosen so that both \textup{\bf(N1)} and
\eqref{eq:Idelta-bound} hold for every $\delta\leq \delta_0$. Then there exist constants $\mu_0,\eta_0>0$ such that for every $\mu\ge \mu_0$ and every $\delta>0$ satisfying
\[
\mu\delta^2\le \eta_0,
\qquad \text{and} \qquad \delta \leq \delta_0
\]
the unique solution of \eqref{eq:model-difference-stoch} satisfies
\begin{equation}\label{eq:main-stoch-est}
    \E\|w_t\|_{\cH}^2
    \leq
    C e^{-\gamma t}\|w_0\|_{\cH}^2
    + C\mu^2 \int_0^t e^{-\gamma (t-s)} \|G_{\delta}(u(s))\|_{\mathrm{L}_2^0}^2\d s,
    \qquad t\geq 0,
\end{equation}
for some positive constants $\gamma= \gamma (\mu, \delta)$ and $C$ independent of $t$.
\end{thm}
\noindent
Estimate \eqref{eq:main-stoch-est} yields exponential decay of the initial
error, up to a contribution from the observational noise. If the noise is
uniformly bounded along the reference trajectory, this contribution remains
uniformly bounded in time and the error decays exponentially towards a
noise floor.
\begin{cor}[Noise floor]\label{cor:noise-floor}
Under the assumptions of \autoref{thm:stoch-H}, assume in addition that
\begin{equation}\label{eq:uniform-noise-bound}
    \Gamma_u:=\sup_{t\geq 0}\|G_{\delta}(u(t))\|_{\mathrm{L}_2^0}^2<\infty.
\end{equation}
Then, for all $t \geq 0$,
\begin{equation}\label{eq:noise-floor-est}
    \E\|w_t\|_{\cH}^2
    \leq
    C e^{-\gamma t}\|w_0\|_{\cH}^2
    + \frac{C\mu^2}{\gamma } \Gamma_u,
\end{equation}
where $\gamma,C>0$ are the constants from \autoref{thm:stoch-H}. In particular,
\begin{equation}\label{eq:noise-floor-limsup}
    \limsup_{t\to\infty}\E\|w_t\|_{\cH}^2 \leq \frac{C\mu^2}{\gamma} \Gamma_u.
\end{equation}
\end{cor}

\begin{cor}[Convergence in the $\cV^\ast$-norm]\label{cor:stoch-Vstar}
Under the assumptions of \autoref{thm:stoch-H}, the unique solution of \eqref{eq:model-difference-stoch} satisfies
\begin{equation}\label{eq:stoch-vstar-est}
    \E\|w_t\|_{\cV^\ast}^2
    \leq
    C e^{-\gamma t}\|w_0\|_{\cH}^2
    + C\mu^2 \int_0^t e^{-\gamma (t-s)} \|G_{\delta}(u(s))\|_{\mathrm{L}_2^0}^2\d s,
    \qquad t\geq 0,
\end{equation}
where $\gamma,C>0$ are the constants from \autoref{thm:stoch-H}. In particular, if \eqref{eq:uniform-noise-bound} holds, then
\[
\limsup_{t\to\infty}\E\|w_t\|_{\cV^\ast}^2 \leq \frac{C\mu^2}{\gamma} \Gamma_u.
\]
\end{cor}
\noindent
The previous bound still depends on the full trajectory $u$ through the quantity $\Gamma_u$. When the deterministic dynamics admit a compact global attractor and the noise coefficient is continuous near that attractor, the asymptotic error can be estimated only in terms of the values of $G_{\delta}$ on the attractor itself.
\begin{cor}[Attractor-dependent noise floor]
Under the assumptions of \autoref{thm:stoch-H}, assume in addition that the
deterministic semiflow generated by \eqref{eq:model} possesses a compact global
attractor $\mathcal{A}\subset\cH$, and that there exists an open neighborhood
$\mathcal U$ of $\mathcal A$ such that
$ G_{\delta}|_{\mathcal{U}}:\mathcal{U}\to \rL_2^0 $ is continuous. Then
$
\Gamma_{\mathcal{A}}:=\sup_{a\in \mathcal{A}} \| G_{\delta}(a) \|_{\rL_2^0}^2<\infty,
$
and
$$
\limsup_{t\to\infty}\mathbb{E}\|w_t\|_{\cH}^2 \leq \frac{C \mu^2}{\gamma} \Gamma_{\mathcal{A}},
$$
where $\gamma,C>0$ are the constants from \autoref{thm:stoch-H}.
\label{cor: attractor noise floor}
\end{cor}
\noindent
Finally, under an additional integrability assumption on the stochastic forcing along the reference trajectory, the mean-square convergence can be upgraded to uniform almost sure convergence. To the best of our knowledge, this is new for stochastic continuous data assimilation with multiplicative noise, except for recent results on data assimilation for the 2D NSEs in a related but different setting \cite{Bessaih2025}.
\begin{thm}[Almost sure convergence in $\cH$]
\label{thm:as-convergence-multiplicative}
Assume the hypotheses of \autoref{thm:stoch-H}, and assume
\begin{equation}
\int_0^\infty \|G_{\delta}( u(t))\|_{\rL_2^0}^2\d t<\infty.
\label{eq:extra-integrability-Gutilde}
\end{equation}
Then
\begin{equation}
    \sup_{t \geq N} \| w_t \|_{\cH} \to 0 \qquad \mathbb{P}\text{- a.s. for } N \to \infty.
    \label{eq: uniform convergence of the tail}
\end{equation}
In particular $
\| w_t\|_{\cH}\to 0 $ $\mathbb{P}$ -a.s. for $t\to\infty$.
\end{thm}

\begin{rem}[Two sufficient conditions for \eqref{eq:extra-integrability-Gutilde}]
Assume the hypotheses of \autoref{thm:stoch-H}. Two natural situations in which condition \eqref{eq:extra-integrability-Gutilde} is satisfied are the following.
\begin{enumerate}
    \item[(a)] Assume that
    $$
    u \in \mathrm{BUC}([0,\infty);\cH) \cap \rL^2(0,\infty;\cH),
    $$
    and that
    $$
    G_{\delta}: \cH \to \rL_2^0
    \qquad\text{is locally Lipschitz-continuous with } G_{\delta}(0)=0.
    $$
    Since the trajectory $u([0,\infty))$ is bounded in $\cH$, the local Lipschitz property of $G_{\delta}$ yields
    $$
    \|G_{\delta}(u(t))\|_{\rL_2^0} \lesssim \|u(t)\|_{\cH},
    \qquad t \geq 0.
    $$
    Hence \eqref{eq:extra-integrability-Gutilde} follows from $u\in \rL^2(0,\infty;\cH)$.

    \item[(b)] Assume that the deterministic semiflow generated by \eqref{eq:model} possesses a compact global attractor $\mathcal{A} \subset \cH$, and that there exist $r>0$ and $L_r>0$ such that $G_{\delta}|_{\mathcal{A}}=0$ and
    $$
    \|G_{\delta}(x)-G_{\delta}(y)\|_{\rL_2^0}\le L_r\|x-y\|_{\cH}
    $$
    whenever $
    \mathrm{dist}_{\cH}(x,\mathcal{A})\leq r$ and $ \mathrm{dist}_{\cH}(y,\mathcal A) \leq r. $
    If moreover
    \begin{equation}
        \int_0^\infty \mathrm{dist}_{\cH}(u(t),\mathcal{A})^2\d t<\infty,
        \label{eq: fast enough conv A}
    \end{equation}
    then \eqref{eq:extra-integrability-Gutilde} holds. Indeed, for all sufficiently large $t$, the trajectory $u(t)$ lies in the $r$-neighborhood of $\mathcal{A}$, and therefore
    $$
    \|G_{\delta}(u(t))\|_{\rL_2^0}\lesssim \mathrm{dist}_{\cH}(u(t),\mathcal{A}).
    $$
    The desired integrability in \eqref{eq:extra-integrability-Gutilde} then follows from \eqref{eq: fast enough conv A}.
\end{enumerate}
\end{rem}

\section{Applications to PDE models}
\label{sec: examples}
\noindent
In this section we illustrate the abstract theory on several PDE models. 
% A key advantage of our Gelfand triple framework is its flexibility: by choosing the triple appropriately, the same approach applies both in weak and in strong settings.

\subsection{Weak formulations}

In this subsection we show that the abstract framework covers several models in their weak formulation. We consider the 2D Navier-Stokes equations, the 2D magnetohydrodynamics (MHD) equations, the 2D quasi-geostrophic equations, and the 1D Allen-Cahn equation. For each model, we verify assumptions \textbf{(A1)}-\textbf{(A4)} and deduce the corresponding mean-square and almost sure convergence results for the associated data-assimilation problem from \autoref{thm:stoch-H} and \autoref{thm:as-convergence-multiplicative}.

\subsubsection{2D Navier-Stokes equations} \mbox{} \\ 
Consider the variational formulation of the Navier-Stokes initial boundary value problem on a bounded domain $\D\subset \R^2$ with a sufficiently smooth boundary $\partial\D$. For this purpose, set
$$
\cV := \rH^1_0 (\D; \mathbb{R}^2) \cap \rL^2_\sigma (\D; \mathbb{R}^2), \qquad \cH := \rL^2_\sigma (\D; \mathbb{R}^2), \qquad \cV^\ast:=\rH^{-1}_\sigma(\D;\mathbb R^2)
:=\cV' ,
$$
where $\cV'$ denotes the dual space of $\cV$. Define the weak Stokes operator $A_w \in \mathcal{L}(\cV, \cV^\ast)$ by
$$
\langle A_w u, \phi \rangle_{\cV^\ast , \cV} := ( \nabla u, \nabla \phi )_2 .
$$
 Defining the bilinear map $B$ by
\begin{equation*}
    \langle B (u,v),\phi \rangle_{\cV^\ast, \cV} := (u\otimes v , \nabla \phi)_2\,,
\end{equation*}
the variational formulation of the Navier-Stokes initial boundary value problem reads as
\begin{equation}\label{eq: variational Stokes}\tag{weak-2D-NSE}
    \left\{
    \begin{aligned}
       u'+A_w u&=B(u,u) , \quad t \in (0,T) ,\\
        u(0)&=u_0.
    \end{aligned}
    \right.
\end{equation}
To verify $\bf(A1)$, we calculate for all $u \in \cV$  in view of Poincar\'e's inequality 
\[
\langle A_w u, u \rangle_{\cV^*,\cV}=\|\nabla u\|_2^2 \geq \alpha \| u \|_{\rH^1}^2 \ \text{ for some }\ \alpha >0.
\]
Concerning $\bf(A2)$, by H\"older's inequality and the Sobolev embedding $\rH^{\frac{1}{2}}(\D)\hookrightarrow \rL^4(\D)$, we obtain
$$
|\langle B(u,v),\phi\rangle_{\cV^*,\cV}|
\leq C\|u\|_4\|v\|_4\|\nabla\phi\|_2
\leq C\|u\|_{\rH^{\frac12}}\|v\|_{\rH^{\frac 12}}\|\phi\|_{\rH^1}.
$$
Hence, $B$ is a bounded bilinear map from $\cV_{\frac34}\times \cV_{\frac34}$ into $\cV^\ast = \rH^{-1}_\sigma (\D)$, since $\cV_{\frac{3}{4}} \hookrightarrow \rH^{\frac{1}{2}}(\D)$. Therefore, \autoref{rem: bilinear} ensures that $\bf(A2)$ holds with $m=1$, $\rho_1=1$, and $\beta_1=\frac34$.

\noindent
Next, we verify \textbf{(A3)}-(i). Setting $
F_w(u):=B(u,u),$ we observe that for every $\phi\in\cV$, an application of the divergence theorem yields
\[
\langle F_w(\phi),\phi\rangle_{\cV^\ast,\cV}
=(\phi\otimes \phi,\nabla \phi)_2
=0.
\]

We finally verify $\bf(A4)$. Let $u$ be the unique global solution to \eqref{eq: variational Stokes}, whose existence is ensured by  $\bf(A1)$-$\bf (A3)$ (see \autoref{prop: global-finite}), and let $\xi\in \cV$. The bilinearity of $B$ gives
$$
B(u,u) - B(u- \xi, u- \xi) = B(\xi, u) + B(u, \xi) - B(\xi, \xi).
$$
Pairing with $\xi$ and using that $u$ and $\xi$ are divergence-free, with zero
trace, it is possible to check that
\[
    \langle B(\xi,u),\xi\rangle_{\cV^\ast,\cV}
    =
    \langle B(\xi,\xi),\xi\rangle_{\cV^\ast,\cV}=0.
\]
Consequently,
\[
    \langle B(u,u)-B(u-\xi,u-\xi),\xi\rangle_{\cV^\ast,\cV}
    =
    (u(t)\otimes \xi,\nabla\xi)_2 .
\]
By H\"older's inequality, Ladyzhenskaya's inequality, and Young's inequality,
for every $\varepsilon_1,\varepsilon_2>0$,
\begin{equation*}
    \begin{split}
        \vert ( u(t) \otimes \xi , \nabla \xi )_2 \vert &\leq \| \nabla \xi \|_2 \|u(t) \|_4 \| \xi \|_4  \leq \eps_1 \| \nabla \xi \|_2^2 + C_{\eps_1} \| u(t) \|_4^2 \| \xi \|_4^2  \\
        &\leq \eps_1 \| \nabla \xi \|_2^2 + C_{\eps_1} \| u(t) \|_2 \| \nabla u (t) \|_2 \| \xi \|_2 \| \nabla \xi \|_2  \\
        & \leq \eps_1 \| \nabla \xi \|_2^2 + \eps_2 \| \nabla \xi \|_2^2 + C_{\eps_1, \eps_2} \| u(t) \|_2^2 \| \nabla u(t) \|_2^2 \| \xi \|_2^2.
    \end{split}
\end{equation*}
Thus, recalling that $F_w(u) = B(u,u),$ condition \eqref{eq:A4} in \textup{\bf (A4)} takes the form
$$
\langle F_w(u(t))- F_w(u(t) - \xi), \xi  \rangle _{\cV^\ast, \cV} \leq (\eps_1 + \eps_2) \| \nabla \xi \|_2^2 + C_{\eps_1, \eps_2} \kappa_u(t) \| \xi \|_2^2,
$$
with $\kappa_u (t) := \| u(t) \|_2^2 \| \nabla u (t) \|_2^2$. Moreover, by the strong energy inequality, see \cite{Prodi},
$$
\frac{1}{2}\| u (\tau_2) \|_2^2 + \int_{\tau_1}^{\tau_2} \| \nabla u (r) \|_2^2 \d r \leq \frac{1}{2}\| u (\tau_1) \|_2^2,
$$
for $0 \leq \tau_1 \leq \tau_2.$ Applying it and H\"older's inequality to estimate $\kappa_u$, we obtain
$$
\int_s^t \kappa_u (r) \d r \leq \sup_{r \leq t} \| u (r) \|_2^2 \int_s^t \| \nabla u (r) \|_2^2 \d r \lesssim  \| u_0 \|_2^2 \| u(s) \|_2^2 \leq \| u_0 \|_2^4,
$$
for $0 \leq s \leq t$. Hence, \eqref{eq:A4avg} also holds, and \textup{\bf (A4)} is verified.

Assuming also \textup{\bf(N1)} and \eqref{eq:Idelta-bound}, we denote by $v$ the unique solution to the stochastic data assimilation problem associated with \eqref{eq: variational Stokes}, see \autoref{sec: preparations}. We consider the difference $\psi=u-v$, which solves
\begin{equation} \label{eq:difference-Stokes}
    \left\{
    \begin{aligned}
        d\psi_t + A_w\psi_t\d t
        &= \bigl(F_w(u(t))-F_w(v_t)\bigr)\d t
        - \mu \, \rI_\delta \psi_t\d t
        - \mu G_{\delta}(u(t))\d W_t^Q,
        \qquad t\in (0,T),\\
        \psi(0)&=\psi_0:=u_0-v_0.
    \end{aligned}
    \right.
\end{equation}
The following convergence
result follows from \autoref{thm:stoch-H}, \autoref{cor:noise-floor}, and
\autoref{thm:as-convergence-multiplicative}.

\begin{cor}[Synchronisation for the weak two-dimensional NSEs]
\label{cor:weak-NSE-synchronisation}
Let $\delta_0>0$ be as in \textup{\bf(N1)} and \eqref{eq:Idelta-bound}.
There exist $\mu_0,\eta_0>0$ such that, for every $\mu\ge\mu_0$ and
$\delta\in(0,\delta_0]$ with $\mu\delta^2\le\eta_0$, the solution
$\psi=u-v$ of \eqref{eq:difference-Stokes} satisfies
\[
    \E\|\psi_t\|_2^2
    \le
    C e^{-\gamma t}\|\psi_0\|_2^2
    + C\mu^2 \int_0^t e^{- \gamma (t-s)} \| G_{\delta}(u(s)) \|^2_{\rL^0_2} \d s,
    \qquad t\ge0,
\]
for some constants $\gamma=\gamma(\mu,\delta)>0$ and $C>0$ independent of $t$.
If $\Gamma_u:= \sup_{t \geq 0}\| G_{\delta} (u(t))\|_{\rL_2^0}^2<\infty$, then
\[
    \E\|\psi_t\|_2^2
    \le
    C e^{-\gamma t}\|\psi_0\|_2^2
    + \frac{C\mu^2}{\gamma}\Gamma_u,
    \qquad
    \limsup_{t\to\infty}\E\|\psi_t\|_2^2
    \le
    \frac{C\mu^2}{\gamma}\Gamma_u .
\]
Finally, if
$
\int_0^\infty \|G_\delta(u(t))\|_{\mathrm L_2^0}^2\d t<\infty,
$
then
\[
    \sup_{t\ge N}\|\psi_t\|_2\to0,
    \qquad \mathbb P\text{-a.s. as }N\to\infty .
\]
\end{cor}

\subsubsection{2D Quasi-geostrophic equations.}\mbox{}\\ Consider the variational formulation of the Quasi-geostrophic initial value problem on the two-dimensional torus $\mathbb T^2$. Set
$$
\cV:=\rH^1(\T^2)\cap \rL^2_0(\T^2),\qquad
\cH:=\rL^2_0(\T^2),\qquad
\cV^\ast:=\cV'
$$
where $\rL^2_0(\T^2)$ denotes the subset of $\rL^2(\T^2)$ of functions having zero spatial average, and $\cV'$ is the dual space of $\cV$. Define the negative weak Laplacian operator $-\Delta_w \in \mathcal{L}(\cV, \cV^\ast)$ by
\[
 \langle -\Delta_w v, \phi\rangle_{\cV^*,\cV}:=(\nabla v,\nabla \phi)_2\,.
\]
We denote by
$$
R^\perp\theta:=\nabla^\perp(-\Delta)^{-1/2}\theta
=
(-R_2\theta,R_1\theta),
$$
where $R_j:=\partial_j(-\Delta)^{-1/2}$, $j=1,2$, are the periodic Riesz transforms on $\T^2$. Since we work on mean-zero functions, this definition is well posed. Moreover, $\div(R^\perp\theta)=0$, and by the Mikhlin theorem $R^\perp$ extends to a bounded linear operator on $\rL^q(\T^2)$ for every $q\in(1,\infty)$. Setting
\[
\langle F_w(\theta),\phi\rangle_{\cV^*,\cV} :=
(\theta\,R^\perp\theta,\nabla\phi)_2,
\]
the variational formulation of the Quasi-geostrophic initial value problem reads as
\begin{equation}\label{eq: variational QG}\tag{weak-2D-QG}
    \left\{
    \begin{aligned}
       \theta'-\Delta_w \theta&=F_w(\theta) , \quad t \in (0,T) ,\\
        \theta(0)&=\theta_0.
    \end{aligned}
    \right.
\end{equation}
To verify $\bf(A1)$, note that, since we are working with functions having zero spatial average, Poincar\'e's inequality yields, for all $\theta\in \cV$,
$$
\langle -\Delta_w \theta,\theta\rangle_{\cV^*,\cV} =
\| \nabla \theta\|_2^2 \geq \alpha \|\theta\|_{\rH^1(\T^2)}^2 = \alpha \|\theta\|_{\cV}^2.
$$
Concerning $\bf(A2)$, define the bilinear map
$$
B(\theta,\eta)\in \cV^*,\qquad \langle B(\theta,\eta),\phi\rangle_{\cV^*,\cV} :=
(\theta\,R^\perp\eta,\nabla\phi)_2,
\qquad \theta,\eta,\phi\in\cV.
$$
Since $R^\perp$ is bounded on $\rL^4(\T^2)$ by the Mikhlin theorem, H\"older's inequality and the Sobolev embedding $\rH^{\frac{1}{2}}(\T^2)\hookrightarrow \rL^4(\T^2)$ imply that
$$
|\langle B(\theta,\eta),\phi\rangle_{\cV^*,\cV}|
\leq C\|\theta\|_4\|R^\perp\eta\|_4\|\nabla\phi\|_2
\leq C\|\theta\|_{\rH^{\frac{1}{2}}}\|\eta\|_{\rH^{\frac{1}{2}}}\|\phi\|_{\rH^1}.
$$
Hence $B: \cV_{\frac{3}{4}} \times \cV_{\frac{3}{4}} \to \cV^\ast $ is bilinear and bounded, where $\cV_{\frac{3}{4}} = \rH^{\frac{1}{2}}(\T^2) \cap \rL_0^2(\T^2)$. Since $F_w(\theta)=B(\theta,\theta)$, \autoref{rem: bilinear} implies that $\bf(A2)$ holds with $m=1$, $\rho_1=1$, and $\beta_1=\frac34$.\\
\noindent
We next verify \textup{\bf(A3)}-(i). Indeed, for all $\phi\in \cV$, by integration by parts on $\T^2$,
$$
\langle F_w(\phi),\phi\rangle_{\cV^*,\cV}
= (\phi\,R^\perp\phi,\nabla\phi)_2
= \frac12(R^\perp\phi,\nabla |\phi|^2)_2
=0,
$$
since $\operatorname{div}(R^\perp\phi)=0$.

We finally verify the validity of $\bf(A4)$ in the form of \eqref{eq:A4}-\eqref{eq:A4avg}. Let $\theta$ be the unique global solution to \eqref{eq: variational QG}, ensured by \autoref{prop: global-finite}, and let $\xi\in \cV$. Using the definition of $F_w$, we compute
\[
\begin{aligned}
\langle F_w(\theta(t))-F_w(\theta(t)-\xi),\xi\rangle_{\cV^\ast,\cV}
&=
(\theta(t)\,R^\perp\xi,\nabla\xi)_2
+
(\xi\,R^\perp\theta(t),\nabla\xi)_2
-
(\xi\,R^\perp\xi,\nabla\xi)_2.
\end{aligned}
\]
The last two terms vanish by integration by parts on $\T^2$ and the identity $\operatorname{div}(R^\perp f)=0$. Therefore
\[
\langle F_w(\theta(t))-F_w(\theta(t)-\xi),\xi\rangle_{\cV^\ast,\cV}
=
(\theta(t)\,R^\perp\xi,\nabla\xi)_2.
\]
Hence, by H\"older's inequality, the boundedness of $R^\perp$ on $\rL^4(\T^2)$, the Gagliardo-Nirenberg inequality, and Young's inequality, for every $\varepsilon>0$ we obtain
\[
\begin{aligned}
|(\theta(t)\,R^\perp\xi,\nabla\xi)_2|
&\le
C\|\theta(t)\|_4\|R^\perp\xi\|_4\|\nabla\xi\|_2 \le
C\|\theta(t)\|_4\|\xi\|_4\|\nabla\xi\|_2 \\
&\le
C\|\theta(t)\|_2^{1/2}\|\nabla\theta(t)\|_2^{1/2}
  \|\xi\|_2^{1/2}\|\nabla\xi\|_2^{3/2} \\
&\le
\varepsilon\|\nabla\xi\|_2^2
+
c\|\theta(t)\|_2^2\|\nabla\theta(t)\|_2^2\|\xi\|_2^2.
\end{aligned}
\]
Therefore \eqref{eq:A4} holds with $
\kappa_\theta(t):= c\|\theta(t)\|_2^2\|\nabla\theta(t)\|_2^2.$ Moreover, by the energy estimate on $\theta$,
$$
\int_s^t \kappa_\theta(\tau)\d \tau \leq
c\sup_{\tau\in[s,t]}\|\theta(\tau)\|_2^2
\int_s^t \|\nabla\theta(\tau)\|_2^2\d \tau
\leq c\|\theta_0\|_2^4,
$$
so \eqref{eq:A4avg} is satisfied with $M_0=0$ and $M_1=c\|\theta_0\|_2^4$.

 Assuming also \textup{\bf(N1)} and \eqref{eq:Idelta-bound}, we denote by $\eta$ the unique solution to the data assimilation problem associated with \eqref{eq: variational QG}. Let $\Theta:=\theta-\eta$ be the assimilation error, that satisfies
\begin{equation} \label{eq:difference-QG} 
    \left\{
    \begin{aligned}
        d\Theta_t - \Delta _w\Theta_t\d t
        &= \bigl(F_w(\theta_t)-F_w(\eta_t)\bigr)\d t
        - \mu \, \rI_\delta \Theta_t\d t
        - \mu G_{\delta}(\theta(t))\d W_t^Q,
        \qquad t\in (0,T),\\
        \Theta(0)&=\Theta_0:=\theta_0-\eta_0.
    \end{aligned}
    \right.
\end{equation}
Therefore, by \autoref{thm:stoch-H}, \autoref{cor:noise-floor}, and
\autoref{thm:as-convergence-multiplicative} the following convergence result holds.
\begin{cor}[Synchronisation for the weak two-dimensional QG equations]
\label{cor:weak-QG-synchronisation}
Let $\delta_0>0$ be as in \textup{\bf(N1)} and \eqref{eq:Idelta-bound}.
There exist $\mu_0,\eta_0>0$ such that, for every $\mu\ge\mu_0$ and
$\delta\in(0,\delta_0]$ with $\mu\delta^2\le\eta_0$, the solution
$\Theta=\theta-\eta$ of \eqref{eq:difference-QG} satisfies
\[
    \E\|\Theta_t\|_2^2
    \le
    C e^{-\gamma t}\|\Theta_0\|_2^2
    + C\mu^2 \int_0^t e^{- \gamma(t-s) }  \| G_\delta (\theta (s)) \|_{\rL_2^0 }^2 \d s,
    \qquad t\ge0,
\]
for some constants $\gamma=\gamma(\mu,\delta)>0$ and $C>0$ independent of $t$.
If $\Gamma_\theta := \sup_{t \geq 0} \| G_{\delta}(\theta (t)) \|_{\rL_2^0}^2 <\infty$, then
\[
    \E\|\Theta_t\|_2^2
    \le
    C e^{-\gamma t}\|\Theta_0\|_2^2
    + \frac{C\mu^2}{\gamma}\Gamma_\theta,
    \qquad
    \limsup_{t\to\infty}\E\|\Theta_t\|_2^2
    \le
    \frac{C\mu^2}{\gamma}\Gamma_\theta .
\]
Finally, if
$
\int_0^\infty \|G_\delta(\theta(t))\|_{\mathrm L_2^0}^2\d t<\infty,
$
then
\[
    \sup_{t\ge N}\|\Theta_t\|_2\to0,
    \qquad \mathbb P\text{-a.s. as }N\to\infty .
\]
\end{cor}

\subsubsection{2D MHD equations.}\mbox{}\\
We consider the initial boundary value problem for the 2D Magnetohydrodynamics equations
(MHD equations) in a bounded simply connected domain $\D\subset\R^2$ with sufficiently
smooth boundary $\partial\D$. We set
$$
\cV : = (\rH^1_{0}(\D;\R^2)\cap \rL^2_\sigma(\D;\R^2)) \times V_2\,,
\qquad \cH : =\rL^2_\sigma(\D;\R^2) \times \rL^2_\sigma(\D;\R^2)\,,
\qquad
\cV^\ast:=\cV',
$$
where, setting
$$
\mathscr{C}(\D):=
\{ h \in \rC^\infty (\overline{\D}; \mathbb{R}^2): \div h=0,\ h \cdot \nu|_{\partial \D}=0\},
$$
we define $V_2$ as the completion of $\mathscr{C}(\D)$ with respect to the
$\rH^1$-norm. In particular, since $\D$ is simply connected, the following
Poincar\'e-type estimate holds
$$
\| h \|_{\rH^1(\D)} \leq C\| \rot h \|_2,
\qquad h \in V_2.
$$
In $\mathbb{R}^2$ we write
$$
\rot a:=\frac{\partial a_2 } {\partial x_1} -\frac{ \partial a_1}{ \partial x_2},
\qquad a=(a_1,a_2)\in \rC^1(\R^2;\R^2).
$$
Let $\Phi=(u,h),$ $ \Psi=(\overline{u},\overline h)\in\cV$. We define the operator
$A_w \in \mathcal{L}(\cV,\cV^\ast)$ as
$$
\langle A_w \Phi, \Psi \rangle_{\cV^*,\cV}
:=(\nabla u, \nabla \overline{u})_2+(\rot h, \rot \overline{h})_2.
$$
Setting
$$
\langle F_w(\Phi), \Psi \rangle_{\cV^*,\cV}
:= (u \otimes u, \nabla \overline{u})_2
-(h \otimes h , \nabla \overline{u})_2
+(h \otimes u,\nabla \overline{h})_2
-(u\otimes h,\nabla \overline{h})_2,
$$
the variational formulation of the MHD initial boundary value problem reads as
\begin{equation}\label{eq: variational MHD}\tag{weak-2D-MHD}
    \left\{
    \begin{aligned}
       \Phi'+A_w \Phi&= F_w(\Phi) , \quad t \in (0,T) ,\\
        \Phi(0)&=\Phi_0.
    \end{aligned}
    \right.
\end{equation}
To verify $\bf(A1)$, we notice that, by virtue of Poincar\'e's inequality and
the above estimate on $V_2$, we have
$$
\langle A_w \Phi, \Phi \rangle_{\cV^*,\cV}
= \| \nabla u \|_2^2+\|\rot h\|_2^2
\geq \alpha\|\Phi\|_{\cV}^2,
$$
for some $\alpha>0$. Concerning $\bf(A2)$, define the bilinear map
$ B: \cV \times \cV \to \cV^\ast $ by
$$
\langle B( \Phi, \Psi), \Pi\rangle_{\cV^*,\cV}:=( u\otimes \overline{u}, \nabla \widetilde{u})_2
-(h \otimes \overline{h}, \nabla \widetilde u)_2   +(h\otimes \overline u, \nabla \widetilde{h})_2
-(u\otimes \overline{h},\nabla \widetilde{h})_2,
$$
for $\Phi=(u,h)$, $\Psi=(\overline u,\overline{h})$, and
$\Pi=(\widetilde u,\widetilde{h})$ in $\cV$. By H\"older's inequality and the
Sobolev embedding $\rH^{\frac{1}{2}}(\D)\hookrightarrow \rL^4(\D)$, we obtain
$$
\vert \langle B(\Phi , \Psi) , \Pi \rangle_{\cV^*,\cV}\vert 
\leq C \| \Phi \|_{\cV_{\frac{3}{4}}} \| \Psi \|_{\cV_{\frac{3}{4}}} \| \Pi \|_{\cV}.
$$
Hence $B$ extends to a bounded bilinear map from
$\cV_{\frac34}\times \cV_{\frac34}$ into $\cV^*$. Since
$F_w(\Phi) = B( \Phi , \Phi)$, \autoref{rem: bilinear} implies that
$\bf(A2)$ holds with $m=1$, $\rho_1=1$, and $\beta_1=\frac{3}{4}$. Moreover,
recalling that given $\Phi=(u,h)\in \cV$, $u$ has zero trace on $\partial\D$
and both $u$ and $h$ are divergence-free, we deduce that
$$
\langle F_w( \Phi ) , \Phi \rangle_{\cV^\ast ,\cV}=0,
$$
and therefore $\bf(A3)$ is fulfilled. \par

We finally show that $\bf(A4)$ is fulfilled. Let $\Phi$ be the unique solution to \eqref{eq: variational MHD} ensured by \autoref{prop: global-finite}. For all
$\Psi=(\overline{u},\overline{h})\in \cV$, in view of the boundary conditions and
since $\Phi$ and $\Psi$ are divergence-free, we find that
$$
\langle F_w(\Phi)-F_w(\Phi-\Psi), \Psi \rangle_{\cV^\ast ,\cV} = (u\otimes \overline{u}, \nabla \overline{u})_2
-(h \otimes \overline{h},\nabla \overline u)_2 
+( h \otimes \overline{u}, \nabla \overline{h})_2
-(u\otimes\overline{h},\nabla\overline{h})_2.
$$
Employing H\"{o}lder's inequality, Ladyzhenskaya's inequality and Young's
inequality, we find that, for every $\varepsilon>0$,
\begin{equation*}
\begin{aligned}
|(u\otimes \overline u,\nabla\overline{u})_2|
&\leq \|\nabla\overline u\|_2\|u\|_4\|\overline{u}\|_4  \\
&\leq \varepsilon \| \nabla \overline{u}\|_2^2
+C_\varepsilon
\| u \|_2^2\|\nabla u\|_2^2\|\overline{u}\|_2^2,
\\[0.3em]
|(h \otimes \overline{h},\nabla\overline u)_2|
&\leq \|\nabla \overline {u}\|_2\|\overline{h}\|_4 \| h \|_4  \\
&\leq \varepsilon (\|\nabla \overline{u}\|_2^2
+\|\nabla \overline{h}\|_2^2 )
+C_\varepsilon \|h\|_2^2\|\nabla h\|_2^2 \| \overline{h} \|_2^2,
\\[0.3em]
\vert (h\otimes \overline{u}, \nabla \overline{h})_2 \vert 
&\leq \|\nabla\overline{h}\|_2\|\overline{u}\|_4 \| h \|_4  \\
&\leq \varepsilon (\|\nabla \overline{u}\|_2^2 +\|\nabla \overline{h}\|_2^2 )
+C_\varepsilon \| h \|_2^2 \| \nabla h\|_2^2\| \overline{u}\|_2^2,
\\[0.3em]
\vert (u\otimes\overline{h},\nabla \overline{h})_2\vert  &\leq \|\nabla \overline{h} \|_2\| u \|_4\| \overline{h} \|_4  \\
&\leq \varepsilon\|\nabla \overline{h}\|_2^2
+C_\varepsilon
\|u\|_2^2\|\nabla u \|_2^2\| \overline{h} \|_2^2.
\end{aligned}
\end{equation*}
Hence, defining
$$
\kappa_\Phi(t) : = C_\varepsilon (
\|u(t)  \|_2^2\|   \nabla u(t)\|_2^2 +\|h(t)\|_2^2\|\nabla h(t)\|_2^2 ),
$$
and relabelling $\varepsilon>0$, we obtain
$$
\vert \langle F_w(\Phi(t)) - F_w (\Phi(t)-\Psi) , \Psi \rangle_{\cV^\ast,\cV} \vert
\leq \varepsilon \| \Psi \|_{\cV}^2
+\kappa_\Phi(t)\|\Psi\|_{\cH}^2.
$$
Moreover, the energy inequality for weak two-dimensional MHD solutions, see
\cite{Temam_MHD}, gives
$$
\sup_{\tau \geq 0}\| \Phi ( \tau ) \|_{\cH}^2 + \int_s^t \| \Phi ( \tau )\|_{\cV}^2 \d \tau \leq C \| \Phi_0\|_{\cH}^2, \qquad 0 \leq s  \leq t.
$$
Consequently,
$$
\int_s^t \kappa_\Phi(\tau)\d \tau \leq C
\sup_{\tau \in[s,t]} \|\Phi(\tau) \|_{\cH}^2 \int_s^t \| \Phi(\tau ) \|_{\cV}^2 \d \tau
\leq C\|\Phi_0\|_{\cH}^4.
$$
Therefore $\bf(A4)$ holds with $M_0=0$ and $M_1=C\|\Phi_0\|_{\cH}^4$.

Assuming that also \textup{\bf(N1)} and \eqref{eq:Idelta-bound} hold, we denote by $\chi$ the unique solution to the stochastic data assimilation problem associated with \eqref{eq: variational MHD}, ensured by the validity of \autoref{prop:DAmult-wp}. Let $\Xi:=\Phi-\chi$ be the assimilation error, that satisfies 
\begin{equation} \label{eq:difference-MHD}
    \left\{
    \begin{aligned}
        d\Xi_t + A_w\Xi_t\d t
        &= \bigl(F_w(\Phi(t))-F_w(\chi_t)\bigr)\d t
        - \mu \, \rI_\delta \Xi_t\d t
        - \mu G_{\delta}(\Phi(t))\d W_t^Q,
        \qquad t\in (0,T),\\
        \Xi(0)&=\Xi_0:=\Phi_0-\chi_0.
    \end{aligned}
    \right.
\end{equation}
Applying \autoref{thm:stoch-H}, \autoref{cor:noise-floor}, and \autoref{thm:as-convergence-multiplicative}, we deduce the following result.
\begin{cor}[Synchronisation for the weak two-dimensional MHD equations]
Let $\delta_0>0$ be as in \textup{\bf(N1)} and \eqref{eq:Idelta-bound}.
There exist $\mu_0,\eta_0>0$ such that, for every $\mu\ge\mu_0$ and
$\delta\in(0,\delta_0]$ with $\mu\delta^2\le\eta_0$, the unique solution $\Xi $ satisfies
\[
    \E\|\Xi_t\|_{\cH}^2
    \le
    C e^{-\gamma t}\|\Xi_0\|_{\cH}^2
    + C\mu^2 \int_0^t e^{- \gamma(t-s) }  \| G_\delta (\Phi (s)) \|_{\rL_2^0 }^2 \d s,
    \qquad t\ge0,
\]
for some constants $\gamma=\gamma(\mu,\delta)>0$ and $C>0$ independent of $t$.
If $\Gamma_\Phi := \sup_{t \geq 0} \| G_{\delta}(\Phi  (t)) \|_{\rL_2^0}^2 <\infty$, then
\[
    \E\|\Xi_t\|_{\cH}^2
    \le
    C e^{-\gamma t}\|\Xi_0\|_{\cH}^2
    + \frac{C\mu^2}{\gamma}\Gamma_\Phi,
    \qquad
    \limsup_{t\to\infty}\E\|\Xi_t\|_{\cH}^2
    \le
    \frac{C\mu^2}{\gamma}\Gamma_\Phi .
\]
Finally, if
$
\int_0^\infty \|G_\delta( \Phi(t))\|_{\mathrm L_2^0}^2\d t<\infty,
$
then
\[
    \sup_{t\ge N}\|\Xi_t\|_{\cH}\to0,
    \qquad \mathbb P\text{-a.s. as }N\to\infty .
\]
\end{cor}

\subsubsection{1D Allen-Cahn Equation.}\mbox{}\\
We consider the variational formulation of the 1D Allen-Cahn initial boundary value problem on $I=(0,1)\subset \R$.
We set
$$
 \cV:=\rH^1_0(I),\qquad \cH:=\rL^2(I),\qquad
 \cV^\ast:=\cV',
$$
where $\cV'$ is the dual space of $\cV$. Let $A_w:= - \partial_{xx}^w \in \mathcal{L}(\cV, \cV^\ast)$ be the negative weak Laplacian, i.e. 
$$
\langle -\partial_{xx}^w u, \phi \rangle_{\cV^*,\cV} = ( \partial_x u , \partial_x \phi)_2 \ \text{ for all } \ u,\phi \in \cV
$$
%\textcolor{red}{G. probably should be $\langle \partial_{xx}^w u, \phi \rangle_{\cV^*,\cV}$ }
and set 
$$\langle F_w(u), \phi \rangle_{\cV^*,\cV} := (u,\phi)_2 - (u^3,\phi)_2 \ \text{ for all } \ u,\phi \in \cV\,.$$
%\textcolor{red}{G. probably should be $\langle \partial_{xx}^w u, \phi \rangle_{\cV^*,\cV}$ }
The variational 1D Allen-Cahn initial boundary value problem reads as
\begin{equation}\label{eq: variational AC}\tag{weak-1D Allen-Cahn}
    \left\{
    \begin{aligned}
       u'+  A_w u&=F_w(u) , \quad t \in (0,T) ,\\
        u(0)&=u_0.
    \end{aligned}
    \right.
\end{equation}
To verify $\bf(A1)$, note that by Poincar\'e's inequality we have
$$
\langle  -\partial_{xx}^w u, u \rangle_{\cV^*,\cV} = \| \partial_x u \|^2_2  \geq \alpha \| u \|^2_{\rH^1} \ \text{ for all } \ u \in \cV\,,
$$
%\textcolor{red}{G. probably should be $\langle \partial_{xx}^w u, \phi \rangle_{\cV^*,\cV}$ }
for a constant $\alpha >0$. Regarding \textup{\bf(A2)}, fix $r\in(1,2)$. Since
$\rL^r(I)\hookrightarrow \rH^{-1}(I)$, H\"older's inequality gives
$$
\|F_w(u)-F_w(v)\|_{\cV^\ast} \leq
C( \|u-v\|_2+\|(u^2+v^2)(u-v)\|_r) \leq
C(\|u-v\|_2+ (\|u\|_{3r}^2+\|v\|_{3r}^2)\|u-v\|_{3r}
).
$$
Moreover, for $r \in (1,2),$ $\cV_{\frac23}=[\cV^\ast,\cV]_{\frac{2}{3}} \hookrightarrow \rH^{\frac13}(I) \hookrightarrow \rL^{3r}(I),$  and $\cV_{\frac23}\hookrightarrow \cH$. Hence
$$
\|F_w(u)-F_w(v)\|_{\cV^\ast} \leq C(
        1+ \| u \|_{\cV_{ \frac{2}{3}}}^2 + \| v \|_{\cV_{ \frac{2}{3}}}^2 )
    \| u-v \|_{\cV_{\frac{2}{3}}} .
$$
Thus \textup{\bf(A2)} holds with $m=1$, $\rho_1=2$, and $\beta_1=\frac{2}{3}$.

\par We now focus on the validity of $\bf(A3)$. Let $u,\phi\in\cV$. Expanding the cubic term,
\begin{equation*}
    \begin{split}
\langle F_w(u+\phi),u\rangle_{\cV^\ast,\cV} &= (u+\phi,u)_2 -((u+\phi)^3,u)_2 \\ &=\|u\|_{2}^2+( \phi,u)_2   -\|u\|_4^4
-3\int_0^1  u^3\phi \d x     -3\int_0^1u^2\phi^2 \d  x
-\int_0^1u\phi^3  \d x\\ &\leq\|u\|_{2}^2+( \phi,u)_2   -\|u\|_4^4
-3\int_0^1  u^3\phi \d x     
-\int_0^1u\phi^3  \d x.
    \end{split}
\end{equation*}
By virtue of Young's inequality, for $\eta\in(0,1)$,
$$
3 \vert u\vert ^3|\phi|\leq \eta \vert u \vert ^4 + C_\eta |\phi|^4,  \quad \vert u  \vert |\phi|^3 \leq \eta |u|^4 + C_\eta |\phi|^4.
$$
Choosing $\eta=\frac14$, we obtain 
$$
\langle F_w(u+\phi),u\rangle_{\cV^\ast,\cV}
\le C\|u\|_{2}^2 + C\|\phi\|_{2}^2 + C\|\phi\|_{\rL^4}^4\,,
$$
which implies
$$
\langle F_w(u+\phi),u\rangle_{\cV^\ast,\cV}
\leq 
 C\,\|u\|_{\cH}^2 + b(\phi),
$$
with
$$
b(\phi):=C(\|\phi\|_{2}^2+\|\phi\|_{4}^4)
\leq C ( \| \phi \|_{2}^2 + \| \phi \|_2^2\| \nabla \phi \|_2^2).
$$
For every $z\in C([0,T];\cH)\cap \rL^2(0,T;\cV)$, we have
$$
\int_0^T b(z(t))\d t\leq C\Bigl(T\sup_{t\in[0,T]}\|z(t)\|_{\cH}^2 +\sup_{t\in[0,T]}\|z(t)\|_{\cH}^2 \int_0^T  \| z(t) \|_{\cV}^2 \d t \Bigr)<\infty.
$$
Therefore \textup{\bf(A3)}-(ii) holds with $\Psi(\phi):=b(\phi).$

\par We now show the validity of $\bf(A4)$. We denote by $u$ the unique global solution to \eqref{eq: variational AC} ensured by the validity of $\bf(A1)$-$\bf(A3)$. Let $\xi\in \cV$. We have 
$$
F_w(u)-F_w(u-\xi)=\xi-(u^3-(u-\xi)^3) =\xi-3u^2\xi+3u\xi^2-\xi^3.
$$
Pairing with $\xi$,
$$
\langle F_w(u)-F_w(u-\xi),\xi\rangle_{\cV^\ast,\cV}
=\int_0^1(1-3u^2)\xi^2 \d x+3\int_0^1u\xi^3 \d x-\int_0^1\xi^4 \d x.
$$
Using $(1-3u^2)\xi^2\le (1+3u^2)\xi^2$ and the pointwise Young inequality
$$
\vert u \vert  \vert \xi\vert ^3 = (|u||\xi|) |\xi|^2 \leq \varepsilon \vert\xi \vert ^4 + C_\varepsilon u^2\xi^2,
$$
we obtain for small $\varepsilon$
$$
\langle F_w(u)-F_w(u-\xi),\xi\rangle_{\cV^\ast,\cV} \leq C(1+\|u\|_{\rL^\infty}^2)\,\|\xi\|_{\cH}^2
-(1-3\varepsilon)\|\xi\|_{\rL^4}^4
\leq C(1+\|u\|_{\rL^\infty}^2)\,\|\xi\|_{\cH}^2.
$$
Now observe that in dimension 1, $\cV = \rH^1_0 \hookrightarrow \rL^\infty $, thus
$$
\|u(t)\|_\infty\leq C \| u (t)\|_{\cV}\,.
$$
Therefore \textup{\bf(A4)} holds with
$$
    \kappa_u(t):=C(1+\|u(t)\|_{\cV}^2).
$$
Indeed, the energy estimate gives
$$
\int_s^t \|u(\tau)\|_{\cV}^2 \d \tau \leq C (1+(t-s)+\|u_0\|_2^2),
$$
and hence
$$
    \int_s^t \kappa_u(\tau) \d \tau \leq M_0(t-s)+M_1
$$
for suitable constants $M_0,M_1 \geq0$.

Assuming also \textup{\bf(N1)} and \eqref{eq:Idelta-bound}, we denote by $v$
the unique solution to the stochastic data assimilation problem associated with
\eqref{eq: variational AC}, and set $\psi:=u-v$. Then the following convergence
result follows from \autoref{thm:stoch-H}, \autoref{cor:noise-floor}, and
\autoref{thm:as-convergence-multiplicative}.

\begin{cor}[Synchronisation for the weak one-dimensional Allen-Cahn equation]
\label{cor:weak-1D-ACE-syncro}
Let $\delta_0>0$ be as in \textup{\bf(N1)} and \eqref{eq:Idelta-bound}.
There exist $\mu_0,\eta_0>0$ such that, for every $\mu\ge\mu_0$ and
$\delta\in(0,\delta_0]$ with $\mu\delta^2\le\eta_0$, the error $\psi$ satisfies
\[
    \E\|\psi_t\|_2^2
    \le
    C e^{-\gamma t}\|\psi_0\|_2^2
    + C\mu^2 \int_0^t e^{- \gamma(t-s) }  \| G_\delta (u (s)) \|_{\rL_2^0 }^2 \d s,
    \qquad t\ge0,
\]
for some constants $\gamma=\gamma(\mu,\delta)>0$ and $C>0$ independent of $t$.
If $\Gamma_u := \sup_{t \geq 0} \| G_{\delta}(u (t)) \|_{\rL_2^0}^2 <\infty$, then
\[
    \E\|\psi_t\|_2^2
    \le
    C e^{-\gamma t}\|\psi_0\|_2^2
    + \frac{C\mu^2}{\gamma}\Gamma_u,
    \qquad
    \limsup_{t\to\infty}\E\|\psi_t\|_2^2
    \le
    \frac{C\mu^2}{\gamma}\Gamma_u .
\]
Finally, if
$
\int_0^\infty \|G_\delta( u(t))\|_{\mathrm L_2^0}^2\d t<\infty,
$
then
\[
    \sup_{t\ge N}\|\psi_t\|_2\to0,
    \qquad \mathbb P\text{-a.s. as }N\to\infty .
\]
\end{cor}

\subsection{Strong formulations}
We now revisit some of the previous models in a stronger functional setting. By selecting a more regular Gelfand triple, the same abstract theory yields convergence results in stronger norms than in the weak framework considered above.

\subsubsection{2D Navier-Stokes equations-revisited}
\label{ex: NSE-strong}\mbox{}\\
We consider the Navier-Stokes equations on the two-dimensional torus $\T^2$. We work with the Gelfand triple
$$
\cV:=\rH^2(\T^2; \mathbb{R}^2) \cap \rL^2_{\sigma,0}(\T^2),  \qquad  \cH:=\rH^1(\T^2;\mathbb R^2)\cap \rL^2_{\sigma,0}(\T^2),
\qquad \cV^\ast:=\rL^2_{\sigma,0}(\T^2),
$$
where $
\rL^2_{\sigma,0}(\T^2)
:= \{
u \in \rL^2(\T^2;\mathbb R^2):
\div u=0,\ \int_{\T^2}u(x)\d x=0 \}.$ Denote by $\P$ the Leray projection onto $\rL^2_{\sigma,0}(\T^2)$, and consider the Stokes operator $A \in \mathcal{L}(\cV, \cV^\ast)$ given by
$$
A:=-\P\Delta,\qquad D(A)=\cV.
$$
We endow $\cH$ and $\cV$ with the equivalent norms
$$
\|u\|_{\cH}:= \| A^{\frac{1}{2}} u  \|_2, \qquad
\|u\|_{\cV}:= \| Au \|_2,
$$
and we recall that the pairing $\langle \cdot ,\cdot \rangle_{\cV^\ast, \cV}$ is given by
$$
\langle f,\phi\rangle_{\cV^\ast,\cV}
:= (f,A\phi)_2, \qquad f\in \rL^2_{\sigma,0}(\T^2),\quad  \phi \in \cV.
$$
Define the bilinear map
$$
B(u,v):=-\P ((u \cdot \nabla) v ), \qquad u,v\in \cV_{\frac{3}{4}},
$$
and set
$$
F(u):=B(u,u)=-\P((u\cdot\nabla)u).
$$
The strong variational formulation of the two-dimensional Navier-Stokes equations reads as
\begin{equation}\label{eq: 2D NSE}\tag{strong-2D-NSE}
    \left\{
    \begin{aligned}
       u'+ A u&=F(u), \quad t\in(0,T),\\
       u(0)&=u_0.
    \end{aligned}
    \right.
\end{equation}
We first verify \textup{\bf(A1)}. By the Poincar\'e inequality on mean-zero
periodic vector fields,
$$
\langle Au,u\rangle_{\cV^\ast,\cV} =
(Au,Au)_2 =
\|Au\|_2^2 \geq \alpha\|u\|_{\rH^2}^2
$$
for some $\alpha>0$ and every $u\in\cV$.  Concerning \textup{\bf(A2)}, since
\[
\cV_{\frac34}=D(A^{\frac34})
=
\rH^{\frac32}(\T^2;\mathbb R^2)\cap \rL^2_{\sigma,0}(\T^2),
\]
and, in dimension two, $\rH^{\frac{3}{2}}(\T^2)\hookrightarrow \rL^\infty(\T^2)$, we have
\[
\|B(u,v)\|_{\cV^\ast}
=
\|\P((u\cdot\nabla)v)\|_2
\leq
\|u\|_\infty\|\nabla v\|_2
\leq
C\|u\|_{\rH^{\frac{3}{2}}}\|v\|_{\rH^{\frac{3}{2}}}.
\]
Thus $B:\cV_{\frac34}\times\cV_{\frac34}\to\cV^\ast$ is bounded and bilinear.
Since $F(u)=B(u,u)$, \autoref{rem: bilinear} implies that
\textup{\bf(A2)} holds with $m=1$, $\rho_1=1$, and $\beta_1= \frac34$.

We next verify \textup{\bf(A3)}-(i). For every $u\in\cV$, using periodic
boundary conditions, the divergence-free constraint, and the classical
two-dimensional cancellation identity,
$$
((u\cdot\nabla)u,Au)_2=0,
$$
see, for instance, \cite{Constantin_Foias_book}, we obtain
$$
\langle F(u), u\rangle_{\cV^\ast,\cV} =
-((u \cdot \nabla) u, Au)_2=0.
$$
Hence \textup{\bf(A3)}-(i) is satisfied.

%\noindent 
We finally verify \textup{\bf(A4)}. Let $u$ be the unique global solution to
\eqref{eq: 2D NSE}, whose existence follows from \textup{\bf(A1)}-\textup{\bf(A3)}, see \autoref{prop: global-finite}. For $\xi\in\cV$, by bilinearity of $B$ we have
$$
F(u)-F(u-\xi)= B( \xi , u )+B(u, \xi ) -B (\xi , \xi).
$$
Pairing with $\xi$ and using again the two-dimensional cancellation
$$
\langle B(\xi,\xi),\xi\rangle_{\cV^\ast,\cV} = -((\xi \cdot \nabla) \xi, A \xi )_2=0,
$$
we obtain
\begin{equation*}
\begin{aligned}
\langle F(u(t))-F(u(t)-\xi), \xi \rangle_{\cV^\ast,\cV} &= \langle B(\xi,u(t)), \xi \rangle_{\cV^\ast,\cV} + \langle B(u(t), \xi ), \xi \rangle_{\cV^\ast,\cV}
\\
&=-(\xi \cdot \nabla u(t), A \xi)_2 -(u(t) \cdot \nabla \xi , A \xi)_2.
\end{aligned}
\end{equation*}
Therefore, by H\"older's and Young's inequalities,
$$
\vert  (u(t) \cdot \nabla \xi, A\xi )_2\vert 
\leq \|u(t)\|_\infty \| \nabla\xi\|_2  \| A \xi \|_2  \leq \varepsilon \| A \xi \|_2^2
+ C_\varepsilon \| u(t) \|_\infty^2 \| \nabla \xi\|_2^2,
$$
while, using $\rH^1(\T^2)\hookrightarrow \rL^6(\T^2)$,
$$
\vert (\xi\cdot \nabla u(t) , A \xi)_2\vert 
\leq \| \xi \|_6 \| \nabla u(t) \|_3 \|A \xi\|_2  \leq \varepsilon \| A \xi\|_2^2
+ C_\varepsilon \| \nabla u(t)\|_3^2 \| \nabla \xi \|_2^2.
$$
Since $\|A\xi\|_2$ is equivalent to $\|\xi\|_{\cV}$ and
$\|\nabla\xi\|_2$ is equivalent to $\|\xi\|_{\cH}$, we infer that
$$
\langle F(u(t))-F(u(t)-\xi), \xi\rangle_{\cV^\ast,\cV} \leq \varepsilon\|\xi\|_{\cV}^2 + C_\varepsilon\kappa_u(t)\|\xi\|_{\cH}^2,
$$
where $ \kappa_u(t):=
\|u(t)\|_\infty^2+\|\nabla u(t)\|_3^2. $
Moreover, by the embeddings $ \rH^2(\T^2) \hookrightarrow \rL^\infty(\T^2),$ $
\rH^2(\T^2)\hookrightarrow  \rW^{1,3}(\T^2),$
we have $
\kappa_u(t)\leq C\|u(t)\|_{\rH^2}^2.$ The strong energy inequality gives
$$
\frac12\|u(t)\|_{\cH}^2 +\int_s^t \|u(r)\|_{\cV}^2 \d r
\leq \frac12\|u(s)\|_{\cH}^2\leq \frac{1}{2} \| u_0 \|_{\cH}^2, 
$$
for $0\leq s\leq t.$ Consequently,
$$
\int_s^t \kappa_u(r) \d r
\leq C\int_s^t \|u(r)\|_{\rH^2}^2 \d r
\leq C \| u_0 \|_{\cH}^2,
$$
and therefore \eqref{eq:A4avg} also holds. Hence \textup{\bf(A4)} is verified.

Assume \textup{\bf(N1)}, \eqref{eq:Idelta-bound}, and denote by $v$ the unique solution to the stochastic data assimilation problem
associated with \eqref{eq: 2D NSE}, whose existence is ensured by \autoref{prop:DAmult-wp}. Set $
w:=u-v.$ Then $w$ solves
\begin{equation}\label{eq:difference-strong-NSE}
    \left\{
    \begin{aligned}
        \d w_t + Aw_t\d t
        &=(F(u(t))-F(v_t)) \d t
        -\mu\,\rI_\delta w_t\d t -\mu G_\delta(u(t)) \d W_t^Q,
        \qquad t>0,\\
        w(0)&=w_0:=u_0-v_0.
    \end{aligned}
    \right.
\end{equation}
The following synchronisation
result follows from \autoref{thm:stoch-H}, \autoref{cor:noise-floor}, and
\autoref{thm:as-convergence-multiplicative}.

\begin{cor}[Synchronisation for the strong two-dimensional NSEs]
\label{cor:strong-NSE-synchronisation}
Let $\delta_0>0$ be as in \textup{\bf(N1)} and \eqref{eq:Idelta-bound}.
There exist $\mu_0,\eta_0>0$ such that, for every $\mu\ge\mu_0$ and
$\delta\in(0,\delta_0]$ with $\mu\delta^2\le\eta_0$, the solution
$w=u-v$ of \eqref{eq:difference-strong-NSE} satisfies
$$
    \E\|w_t\|_{\rH^1}^2
    \le
    C e^{-\gamma t}\|w_0\|_{\rH^1}^2
    + C\mu^2 \int_0^t e^{- \gamma (t-s)} \| G_{\delta}(u(s)) \|^2_{\rL^0_2} \d s,
    \qquad t \geq 0,
$$
for some positive $\gamma=\gamma(\mu,\delta)>0$ and $C>0$ independent of $t$. If $\Gamma_u:= \sup_{t \geq 0}\| G_{\delta} (u(t))\|_{\rL_2^0}^2<\infty$, then
$$
    \E \|w_t\|_{\rH^1}^2 \leq
    C e^{-\gamma t} \| w_0 \|_{\rH^1}^2
    + \frac{C\mu^2}{\gamma}\Gamma_u, \qquad
    \limsup_{t\to\infty} \E \| w_t \|_{\rH^1}^2
    \leq \frac{C\mu^2}{\gamma}\Gamma_u .
$$
Finally, if $\int_0^\infty \|G_\delta(u(t))\|_{\rL_2^0}^2 \d t<\infty,$
then
$$
    \sup_{t \geq N}\| w_t\|_{\rH^1} \to 0,
    \qquad \mathbb{P} \text{-a.s. as }N \to \infty .
$$
\end{cor}

\subsubsection{1D Allen-Cahn equation-revisited}\mbox{}\\
We consider the initial boundary value problem for the Allen-Cahn equation on the real interval $I=(0,1)$. We work with the Gelfand triple
$$
\cV:=\rH^2(I)\cap \rH^1_0(I), \qquad \cH:=\rH^1_0(I), \qquad \cV^\ast:=\rL^2(I).
$$
and let $A=-\partial_{xx} \in \mathcal{L} (\cV, \cV^\ast )$ be the negative Dirichlet Laplacian on $\rL^2(I)$. We endow $\cV, \cH$ with the equivalent norms
$$
\|u\|_{\cV}:=\|Au\|_2=\|\partial_{xx}u\|_2, \qquad
\|u\|_{\cH}:=\|A^{\frac{1}{2}}u\|_2=\|\partial_xu\|_2.
$$
The duality pairing between $\cV^*$ and $\cV$ is defined by
$$
\langle f,\phi\rangle_{\cV^*,\cV}:=(f,A\phi)_2
=(f,-\partial_{xx}\phi)_2,
\qquad f\in \rL^2(I),\ \phi\in\cV.
$$
The initial boundary value problem for the one-dimensional Allen-Cahn equation reads as
\begin{equation}\label{eq: 1D Allen-Cahn}\tag{1D Allen-Cahn}
    \left\{
    \begin{aligned}
       u'-\partial_{xx} u&=F(u) , \quad t \in (0,T) ,\\
        u(0)&=u_0,
    \end{aligned}
    \right.
\end{equation}
where $F(u):=u-u^3$. We first verify $\bf(A1)$. For every $u\in\cV$,
$$
\langle Au,u\rangle_{\cV^*,\cV} =(Au,Au)_2 =\|u\|_{\cV}^2,
$$
so $\bf(A1)$ holds with $\alpha=1$. Concerning $\bf (A2)$, for all $u,v\in \cV$ we have
\begin{equation*}
    \begin{split}
        \| F(u) - F(v) \|_2 &= \| u -v - (u^3 - v^3) \|_2 \leq \| u-v \|_2 + \| (u-v) (u^2 +uv+v^2) \|_2 \\
        &\leq C (1+ \| u\|_\infty^2 + \|v\|_\infty^2) \| u-v\|_{\cV_\beta}
    \end{split}
\end{equation*}
Since in one dimension it holds $\cV_{\beta} = [\cV^\ast, \cV]_{\beta} \hookrightarrow\rH^{2\beta}(I) \hookrightarrow \rL^\infty(I)$ for $\beta > \frac{1}{4},$ then $\bf (A2)$ is satisfied with $m =1$, $\rho = \rho_1 = 2$, and for any $\beta= \beta_1 \in (\frac{1}{2}, \frac{2}{3}],$ because \eqref{eq:A3_coefficients_bound} is equivalent to $\beta \leq \frac{2}{3}.$

We now verify $\bf(A3)$. Let $x,y\in\cV$ and set $ s:=x+y. $
Using the duality pairing and integrating by parts, we obtain
\begin{equation*}
    \begin{split}
        \langle F(x+y) , x \rangle_{\cV^*,\cV} &=(s-s^3, Ax)_2                                              = \int_0^1 (s-s^3 )(-x'') \d r                                =\int_0^1 (1-3s^2)s'x' \d r                                      \\
&=\| x \|_{\cH}^2+(y,x)_{\cH}
   -3\int_0^1 s^2|x'|^2 \d r
   -3\int_0^1 s^2 y' x' \d r .
    \end{split}
\end{equation*}
The last two terms are treated by completing the square
$$
-3s^2|x'|^2-3s^2y'x'=
-3s^2  ( x'+\frac{1}{2}y' )^2
+\frac34s^2 (y')^2                                      \leq
\frac{3}{4} s^2 |y'|^2 .
$$
Therefore, recalling that $s = x+y$, we obtain
$$
\begin{aligned}
\langle F(x+y),x\rangle_{\cV^*,\cV} \leq
\|x\|_{\cH}^2
+ |(y,x)_{\cH}|
+ C\int_0^1 |x+y|^2|y'|^2 \d r                                     \leq
C\|x\|_{\cH}^2
+C\|y\|_{\cH}^2
+C\|x+y\|_\infty^2\|y\|_{\cH}^2 .
\end{aligned}
$$
Since in one space dimension it holds the embedding $
\cH=\rH^1_0(I)\hookrightarrow \rL^\infty(I),$
it follows 
\[
\|x+y\|_\infty^2
\leq
C\bigl(\|x\|_{\cH}^2+\|y\|_{\cH}^2\bigr).
\]
Hence
\[
\begin{aligned}
\langle F(x+y),x\rangle_{\cV^*,\cV}
&\leq
C\|x\|_{\cH}^2
+C\|y\|_{\cH}^2
+C\bigl(\|x\|_{\cH}^2+\|y\|_{\cH}^2\bigr)\|y\|_{\cH}^2                 \leq
C\bigl(1+\|y\|_{\cH}^2\bigr)\|x\|_{\cH}^2
+C\bigl(\|y\|_{\cH}^2+\|y\|_{\cH}^4\bigr).
\end{aligned}
\]
Thus $\bf(A3)$ holds with $
\Psi(y):=
C\bigl(\|y\|_{\cH}^2+\|y\|_{\cH}^4\bigr).$
In particular, if $ z\in \rC([0,t];\cH)\cap \rL^2(0,t;\cV),$
then $ \Psi(z(\cdot))\in \rL^1(0,t)$.

We finally verify the validity of $\bf(A4)$ by proving a stronger one-sided
increment estimate. Let $x,y\in\cV$. Since
$$
F(x+y)-F(x)=y-3x^2y-3xy^2-y^3,
$$
we have, using the duality pairing and integrating by parts,
\begin{equation*}
    \begin{split}
        \langle F(x+y)-F(x),y\rangle_{\cV^*,\cV}
&=(F(x+y)-F(x),Ay)_2                               =\int_0^1 \partial_r (F(x+y)-F(x) )y' \d r             \\
&=\|y\|_{\cH}^2 -3\int_0^1 (x+y)^2|y'|^2 \d r                     -3\int_0^1 y(2x+y)x'y' \d r .
    \end{split}
\end{equation*}
Since $ 2x+y=(x+y)+x,$
we can write
$$
-3\int_0^1 y(2x+y)x'y' \d r = -3\int_0^1 y(x+y)x'y' \d r -3\int_0^1 xyx'y'\,\d r .
$$
The first term is absorbed by the negative cubic contribution. Indeed, by
Young's inequality,
$$
3|y|\, |x+y|\,|x'|\,|y'|
\leq \frac{3}{2} (x+y)^2|y'|^2 +C|x'|^2|y|^2.
$$
Hence
$$
-3\int_0^1 (x+y)^2|y'|^2 \d r -3\int_0^1 y(x+y)x'y' \d r                         \leq C\int_0^1 |x'|^2|y|^2 \d r .
$$
For the remaining term, again by Young's inequality,
$$
3|x|\,|y|\,|x'|\,|y'|\leq \frac{1}{2} |y'|^2
+
C|x|^2|x'|^2|y|^2.
$$
Consequently,
$$
\langle F(x+y)-F(x),y\rangle_{\cV^*,\cV} \leq
C\|y\|_{\cH}^2 +C\int_0^1 |x'|^2|y|^2 \d r
+C\int_0^1 |x|^2|x'|^2|y|^2\,\d r .
$$
Using the one-dimensional estimates
$$
\|x\|_\infty\leq C\|x\|_{\cH},
\qquad
\|x'\|_\infty\leq C\|x\|_{\cV},
\qquad
\|y\|_2\leq C\|y\|_{\cH},
$$
we infer
$$
\langle F(x+y)-F(x),y\rangle_{\cV^*,\cV}
\leq
C\|y\|_{\cH}^2
+C\|x\|_{\cV}^2\|y\|_{\cH}^2
+C\|x\|_{\cH}^2\|x\|_{\cV}^2\|y\|_{\cH}^2           \leq
C (1+\widetilde \Psi(x))\|y\|_{\cH}^2,
$$
where $ \widetilde \Psi(x):=(1+\|x\|_{\cH}^2)\|x\|_{\cV}^2 .$
Therefore, taking $x=u(t)$ and $y=-\xi$, we obtain
$$
\langle F(u(t))-F(u(t)-\xi),\xi\rangle_{\cV^*,\cV}
= \langle F(u(t)-\xi)-F(u(t)),-\xi\rangle_{\cV^*,\cV}        \leq
C (1+ \widetilde \Psi(u(t)))\|\xi\|_{\cH}^2 .
$$
Thus \eqref{eq:A4} holds  with $\kappa_u(t):=C(1+\widetilde\Psi(u(t))).$

It remains to verify \eqref{eq:A4avg}. Testing
\eqref{eq: 1D Allen-Cahn} against $Au$ gives
$$
\frac12\frac{\d}{\d t}\|u(t)\|_{\cH}^2+\|u(t)\|_{\cV}^2 = (F(u(t)),Au(t))_2 .
$$
Moreover,
$$
(F(u),Au)_2 =
(u-u^3,Au)_2
= \|u\|_{\cH}^2-3\int_0^1 u^2|u'|^2\,\d r
\leq \|u\|_{\cH}^2.
$$
By the Poincar\'e inequality,
$$
\|u\|_{\cH}^2\leq \lambda_1^{-1}\|u\|_{\cV}^2,
\qquad \lambda_1=\pi^2,
$$
and therefore
$$
\frac{1}{2}\frac{\d}{\d t}\|u(t)\|_{\cH}^2
+ (1-\lambda_1^{-1} ) \| u(t) \|_{\cV}^2 \leq 0.
$$
Hence $
\sup_{t\geq0}\|u(t)\|_{\cH}^2 \leq \|u_0\|_{\cH}^2 $
and
$$
\int_s^t \|u(r)\|_{\cV}^2 \d r \leq C\|u(s)\|_{\cH}^2
\leq C\|u_0\|_{\cH}^2, \qquad 0\leq s\leq t.
$$
Consequently,
$$
\int_s^t \kappa_u(r) \d r\leq C(t-s) +C(1+\sup_{r\geq0}\|u(r)\|_{\cH}^2) \int_s^t \|u(r)\|_{\cV}^2 \d r                               \leq C(t-s)+C_{u_0},
$$
for any $0 \leq s \leq t$.
This proves \eqref{eq:A4avg}.

Assume moreover that \textup{\bf(N1)} and \eqref{eq:Idelta-bound} hold. Let $v$ be the solution to the data assimilation problem related to
\eqref{eq: 1D Allen-Cahn}. Setting $w:=u-v$, we have
\begin{equation}\label{eq:difference-strong-AC}
    \left\{
    \begin{aligned}
        \d w_t + Aw_t \d t
        &=(F(u(t))-F(v_t))\d t-\mu\,\rI_\delta w_t \d t-\mu G_\delta(u(t)) \d W_t^Q,
        \qquad t>0,\\
        w(0)&=w_0:=u_0-v_0.
    \end{aligned}
    \right.
\end{equation}
From \autoref{thm:stoch-H}, \autoref{cor:noise-floor}, and
\autoref{thm:as-convergence-multiplicative}, we obtain the following convergence result.
\begin{cor}[Synchronisation for the strong one-dimensional Allen-Cahn equation]
\label{cor:strong-AC-synchronisation}
Let $\delta_0>0$ be as in \textup{\bf(N1)} and \eqref{eq:Idelta-bound}.
There exist $\mu_0,\eta_0>0$ such that, for every $\mu\ge\mu_0$ and
$\delta\in(0,\delta_0]$ with $\mu\delta^2\le\eta_0$, the solution
$w=u-v$ of \eqref{eq:difference-strong-AC} satisfies
$$
    \E\|w_t\|_{\rH^1}^2
    \le
    C e^{-\gamma t}\|w_0\|_{\rH^1}^2
    + C\mu^2 \int_0^t e^{- \gamma (t-s)} \| G_{\delta}(u(s)) \|^2_{\rL^0_2} \d s,
    \qquad t \geq 0,
$$
for some constants $\gamma=\gamma(\mu,\delta)>0$ and $C>0$ independent of $t$. If $\Gamma_u:= \sup_{t \geq 0}\| G_{\delta} (u(t))\|_{\rL_2^0}^2<\infty$, then
$$
    \E \|w_t\|_{\rH^1}^2 \leq
    C e^{-\gamma t} \| w_0 \|_{\rH^1}^2
    + \frac{C\mu^2}{\gamma}\Gamma_u, \qquad
    \limsup_{t\to\infty} \E \| w_t \|_{\rH^1}^2
    \leq \frac{C\mu^2}{\gamma}\Gamma_u .
$$
Finally, if $\int_0^\infty \|G_\delta(u(t))\|_{\rL_2^0}^2 \d t<\infty,$
then
$$
    \sup_{t \geq N}\|w_t\|_{\rH^1} \to 0,
    \qquad \mathbb{P} \text{-a.s. as }N \to \infty .
$$
\end{cor}

\section{Preparatory well-posedness results}
\label{sec: preparations}
\noindent
In this section we prepare the proofs of the main results. We first show that, under $\bf{(A1)-(A3)}$, the deterministic reference problem \eqref{eq:model} is globally well posed on every finite time interval. We then turn to the stochastic data-assimilation system and reduce it, via the Da Prato-Debussche decomposition, to a pathwise random PDE. This allows us to transfer the deterministic well-posedness argument to the stochastic setting.

Based on the assumptions $\bf(A1)-(A2)$, we recall the following result on the local well-posedness of \eqref{eq:model}, see \cite[Theorem 18.2.6 and Theorem 18.2.17]{Veraar3}.
\begin{lem}\label{lem: local}
    Let $T>0$ and assume that $\bf(A1)-\bf(A2)$ are satisfied. Then for any $u_0 \in \cH$, there exists $a = a(u_0)\leq T$ such that problem \eqref{eq:model} admits a unique solution
    \[
        u \in \rL^2(0,a;\cV) \cap \rH^1(0,a;\cV^\ast) \cap \mathrm{BUC}([0,a];\cH).
    \]
    The solution exists on a maximal time interval $[0,a_{\mathrm{max}}(u_0))$ and depends continuously on the data. If the solution does not exist globally in time, i.e., if $a_{\mathrm{max}} <T$, then the maximal existence time is characterized by the blow-up
    \begin{equation}\label{eq: blow up}
          \lim\limits_{t \to a_{\mathrm{max}}} \| u \|_{ \rL^2(0,t;\cV) \cap \rH^1(0,t;\cV^\ast) } = \infty.
    \end{equation}
\end{lem}
\noindent
Adding the further non-blow up assumption \textbf{(A3)}, the first consequence is that the deterministic reference problem is globally well posed on every finite time interval.
\begin{prop}[Global well-posedness on finite time intervals]\label{prop: global-finite}
Assume that $\bf(A1)$--$\bf(A3)$ hold. Then for every $u_0\in \cH$ and every $T>0$, the problem \eqref{eq:model} admits a unique solution
\[
    u \in \rL^2(0,T;\cV)\cap \rH^1(0,T;\cV^\ast)\cap \mathrm{BUC}([0,T];\cH).
\]
In particular, \eqref{eq:model} is globally well-posed on every finite time interval.
\end{prop}
\begin{proof}[Proof of \autoref{prop: global-finite}]
By \autoref{lem: local}, there exists a unique maximal solution $u$
on some interval $[0,a_{\max})\subseteq [0,T]$. Taking inner products of \eqref{eq:model} by $u$ in $\cV$ and using $\bf(A1)$, we obtain
\[
\frac12 \frac{\d}{\d t} \|u(t)\|_{\cH}^2 + \alpha\|u(t)\|_{\cV}^2
=
\langle F(u(t)),u(t)\rangle_{\cV^\ast,\cV}.
\]
Since $F=\sum_{j=1}^m F_j$, we estimate the components separately. For those $j$ with $\beta_j\geq \frac34$, we use \textup{\bf(A3)}-(i), while for those
$j$ with $\beta_j<\frac34$, we use \textup{\bf(A3)}-(ii) with $y=0$. Hence, setting $\varepsilon_* := \sum_{j=1}^m \eps_j$,
\[
    \langle F(u(t)),u(t)\rangle_{\cV^\ast,\cV}
    \leq \eps_* \|u(t)\|_\cV^2
    + C\|u(t)\|_\cH^2
    + C .
\]
Since $\eps_* \in (0 , \tfrac{\alpha}{4})$ by \textup{\bf (A3)}, it follows that
\[
    \frac12\frac{\d}{\d t}\|u(t)\|_\cH^2
    +c_0\|u(t)\|_\cV^2
    \le
    C\|u(t)\|_\cH^2+C,
\]
with $c_0 := \alpha - \eps_*>0. $ Gronwall's lemma gives $u\in \rL^\infty(0,a_{\max};\cH)\cap \rL^2(0,a_{\max};\cV)$
with bounds depending only on $T$, $\|u_0\|_{\cH}$ and the constants in $\bf(A1)$ and $\bf(A3)$. To show that $u'\in \rL^2(0,a_{\max};\cV^\ast)$ we write 
\begin{equation*}
    u' = -Au + F(u)
\end{equation*}
and notice that $Au\in \rL^2(0,a_{\max};\cV^\ast)$. So it remains to show that $F(u)\in \rL^2(0,a_{\max};\cV^\ast)$.
By \eqref{eq:A2_Fj_bound} in $\bf(A2)$ with $v=0$, we obtain 
\[
\|F(u)\|_{\rL^2(0,a_{\max};\cV^\ast)}^2
\le
C \| F(0) \|_{\cV^\ast}^2
+
C\sum_{j=1}^m
\int_0^{a_{\max}} 
\|u(s)\|^2_{\cV_{\beta_j}}
+
\|u(s)\|_{\cV_{\beta_j}}^{2(\rho_j+1)}
\d s.
\]
Since $\cV\hookrightarrow \cV_{\beta_j}$, we immediately have $ u\in \rL^2(0,a_{\max};\cV_{\beta_j}).$ Moreover, setting \(\theta_j:=2\beta_j-1\), interpolation gives
$$
\|u\|_{\cV_{\beta_j}}
\le C\|u\|_{\cH}^{1-\theta_j}\|u\|_{\cV}^{\theta_j}.
$$
Hence, using \(u\in \rL^\infty(0,a_{\max};\cH)\cap
\rL^2(0,a_{\max};\cV)\) and
\(\theta_j(\rho_j+1)\le1\), which holds thanks to \eqref{eq:A3_coefficients_bound} in \textbf{(A2)}, we obtain
\[
u\in \rL^{2(\rho_j+1)}(0,a_{\max};\cV_{\beta_j}).
\]
Therefore, \(F(u)\in \rL^2(0,a_{\max};\cV^*)\). Consequently,
\[
\sup_{t<a_{\max}} \|u\|_{\rL^2(0,t;\cV)\cap \rH^1(0,t;\cV^\ast)} <\infty.
\]
The blow-up alternative \eqref{eq: blow up} cannot occur at any finite time, and therefore $a_{\max}=T$.
\end{proof}
\noindent 
We now turn to the stochastic data-assimilation problem \eqref{eq:model-data-stoch-u}. The key point is to separate the multiplicative stochastic forcing from the nonlinear dynamics by means of the Da Prato-Debussche decomposition. More precisely, for the reference solution $u$ given by \autoref{prop: global-finite}, we first introduce the stochastic convolution $Z$. Once $Z$ is constructed, the unknown $v$ is written as $v=\hat v+Z$, and the stochastic equation is reduced to a pathwise random PDE for $\hat v$. The latter can then be treated by the same energy method used above for the deterministic problem. For $T>0$, consider
\begin{equation}\label{eq:stoch-conv}
    \left\{
    \begin{aligned}
        \d Z_t + AZ_t\d t &= \mu \,G_{\delta}(u(t))\d W_t^Q,
        \qquad t\in (0,T),\\
        Z(0)&=0.
    \end{aligned}
    \right.
\end{equation}
The next lemma yields existence and maximal-regularity estimates for the stochastic convolution solving \eqref{eq:stoch-conv}. This result is a direct consequence of the stochastic maximal $\rL^2$-regularity of $A$, recalled in \autoref{rem:A-stoch-max-reg}, and which follows from \textup{\bf(A1)}.
\begin{lem}[Existence and regularity of the stochastic convolution]\label{lem:stoch-conv}\phantom{ }\\
Let $T >0$. Then, the auxiliary problem \eqref{eq:stoch-conv} admits a unique mild solution $Z$, given by 
\[
Z(t)=\mu\int_0^t e^{-(t-s)A}G_{\delta}(u(s))\d W_s^Q,
\qquad t\in [0,T],
\]
such that $Z$ is adapted as an $\cH$-valued process and progressively measurable as a $\cV$-valued process. Moreover, there exists a constant $c>0$, independent of $T$, such that
\begin{equation}\label{eq:stoch-max-reg}
    \E \|Z\|_{\rL^2(0,T;\cV)}^2
    + \E \|Z\|_{C([0,T];\cH)}^2
    \leq c  \mu^2 \int_0^T \|G_{\delta}(u(t))\|_{\mathrm{L}_2^0}^2\d t.
\end{equation}
In particular,
\[
Z\in \rC([0,T];\cH)\cap \rL^2(0,T;\cV)
\qquad \mathbb{P}\text{-a.s.}
\]
\end{lem}
\noindent
We stress that $Z$ (and the norm of $Z$ in this space) depends on $\delta$ but we have suppressed this in our notation as the precise dependence is not relevant.

We now perform the Da Prato-Debussche decomposition at the level of the original variable $v$. Let $T>0$ and let $Z$ be the solution of \eqref{eq:stoch-conv}. Define
\[
\hat v:=v-Z.
\]
Then $\hat v$ satisfies the random PDE
\begin{equation}\label{eq:pathwise-random-pde}
    \left\{
    \begin{aligned}
        \hat v'(t)+A\hat v(t)
        &= F\bigl(\hat v(t)+Z(t)\bigr)
        -\mu \, \rI_\delta \hat v(t)
        +\mu \, \rI_\delta\bigl(u(t)-Z(t)\bigr),
        \qquad t\in (0,T),\\
        \hat v(0)&=v_0.
    \end{aligned}
    \right.
\end{equation}
This reformulation isolates the random forcing in the path $Z$ and reduces the analysis to a deterministic problem with random coefficients. The next proposition shows that, once this reduction is made, the deterministic well-posedness argument can be adapted pathwise to prove well-posedness of the stochastic data-assimilation system on finite intervals whenever $\mu\delta^2 \lesssim 1$.
\begin{prop}[Well-posedness of the stochastic data-assimilation problem on finite intervals]
\label{prop:DAmult-wp}
Assume that $\rI_\delta$ satisfies \eqref{eq:Idelta-bound}, and let $v_0\in \cH$. Suppose that $\bf(A1)$--$\bf(A3)$ and $\bf(N1)$ hold. Let $C_{\rI}>0$ denote the constant in \eqref{eq:Idelta-bound}, $\alpha>0$ the coercivity constant from $\bf(A1)$ and $\delta_0$ the constant involved in $\bf (N1)$ and \eqref{eq:Idelta-bound}. Set
\[
\eta_0:=\frac{2\alpha}{C_{\rI}^2}.
\]
Then, for every $T>0$, for every $\mu>0$ and every $0<\delta\le \delta_0$ satisfying
\[
\mu\delta^2\le \eta_0,
\]
the stochastic data-assimilation problem \eqref{eq:model-data-stoch-u} admits a unique solution $v$ on $[0,T]$, such that $v$ is adapted as an $\cH$-valued process and progressively measurable as a $\cV$-valued process, and
\[
v-Z \in \rL^2(0,T;\cV)\cap \rH^1(0,T;\cV^\ast)\cap \mathrm{BUC}([0,T];\cH)
\qquad \mathbb{P}\text{-a.s.},
\]
In particular
\[
v\in \rC([0,T];\cH)\cap \rL^2(0,T;\cV)
\qquad \mathbb{P}\text{-a.s.}
\]
\end{prop}
\begin{proof}[Proof of \autoref{prop:DAmult-wp}]
We divide the proof in three steps.

\emph{Step $1$: reduction to pathwise problem and local in time existence.} Fix $T>0$. By \autoref{lem:stoch-conv}, the stochastic convolution $Z$ belongs to
\[
\rC([0,T];\cH)\cap \rL^2(0,T;\cV)
\qquad \mathbb{P}\text{-a.s.}
\]
Hence, there exists a set $\Omega_0\in \mathcal F$ with $\mathbb{P}(\Omega_0)=1$ such that, for every $\omega\in \Omega_0$, the path
\[
z:=Z(\cdot,\omega)
\]
belongs to $\rC([0,T];\cH)\cap \rL^2(0,T;\cV)$.
Fix $\omega\in \Omega_0$. Then $\hat v(\cdot,\omega)$ solves the deterministic problem
\begin{equation}\label{eq:pathwise-frozen}
    \left\{
    \begin{aligned}
        y'(t)+Ay(t)
        &=F\bigl(y(t)+z(t)\bigr)-\mu \, \rI_\delta y(t)
        +\mu \, \rI_\delta\bigl(u(t)-z(t)\bigr),
        \qquad t\in (0,T),\\
        y(0)&=v_0.
    \end{aligned}
    \right.
\end{equation}
Equivalently, setting $
\widetilde{A}_{\mu, \delta} := A + \mu \rI_{\delta},$ and $G_z(t,x) := F(x+ z(t)) + \mu \rI_{\delta}( u(t) - z(t)),$ the non-autonomous semilinear problem \eqref{eq:pathwise-frozen} takes the form
\begin{equation}
\left \lbrace
    \begin{aligned}
        y'(t) + \widetilde{A} y(t) &= G_z(t, y(t)), \qquad t \geq 0 \\
        y(0) & = v_0.
    \end{aligned}
    \right. 
\end{equation}
In order to investigate the well-posedness of this problem, we employ \autoref{lemma: new non autonomous lemma local existence}. Hence, we first need to show that its assumptions are fulfilled. First, for the $\rL^2$-maximal regularity, observe that, using \textup{\bf(A1)}, \eqref{eq:Idelta-bound}, and Young's inequality, for every $\xi\in \cV$ we have
\[
\begin{aligned}
  \langle \widetilde{A}\xi,\xi\rangle_{\cV^\ast,\cV}=\langle (A+\mu \, \rI_\delta)\xi,\xi\rangle_{\cV^\ast,\cV}
&=
\langle A\xi,\xi\rangle_{\cV^\ast,\cV}
+\mu \langle \rI_\delta \xi-\xi,\xi\rangle_{\cV^\ast,\cV}
+\mu \|\xi\|_\cH^2 \\
&\ge
\alpha \|\xi\|_\cV^2
-
C_{\rI}\mu\delta \|\xi\|_\cH\|\xi\|_\cV
+
\mu \|\xi\|_\cH^2 \\
&\ge
\frac{\alpha}{2}\|\xi\|_\cV^2
+
\left(\mu-\frac{C_{\rI}^2\mu^2\delta^2}{2\alpha}\right)\|\xi\|_\cH^2.
\end{aligned}
\]
If $\mu\delta^2\le \eta_0=2\alpha/C_{\rI}^2$, then the coefficient of $\|\xi\|_\cH^2$ is nonnegative, and therefore
\begin{equation}\label{eq:coercive-Amudelta}
\langle\widetilde{A} \xi, \xi  \rangle_{\cV^\ast,\cV} = \langle (A+\mu \, \rI_\delta)\xi,\xi\rangle_{\cV^\ast,\cV}
\ge
\frac{\alpha}{2}\|\xi\|_\cV^2,
\qquad \forall\, \xi\in \cV.
\end{equation}
In particular, the operator $A+\mu \, \rI_\delta$ is coercive on $\cV$ and enjoys $\rL^2$-maximal regularity.

Concerning the assumption (2) in \autoref{lemma: new non autonomous lemma local existence}, for a.e. $t \in (0,T)$ and for any $x,y \in \cV$, we have to estimate
$$
\| G_z(t,x) - G_z(t,y) \|_{\cV^\ast } = \| F (x+ z(t)) - F(y+z(t)) \|_{\cV^\ast }.
$$
By \eqref{eq:A2_Fj_bound}, it holds
\begin{equation*}
    \begin{split}
        \| G_z(t,x) - G_z(t,y) \|_{\cV^\ast }& \leq C \sum_{j =1}^m (1 + \| x+ z(t) \|_{\cV_{\beta_j}}^{\rho_j} + \| y + z(t) \|_{\cV_{\beta_j}}^{\rho_j}) \| x- y \|_{\cV_{\beta_j}} \\
        & \leq C \sum_{j =1}^m (1 + \| x \|_{\cV_{\beta_j}}^{\rho_j}+ \| z(t) \|_{\cV_{\beta_j}}^{\rho_j} + \| y\|_{\cV_{\beta_j}}^{\rho_j} + \|z(t) \|_{\cV_{\beta_j}}^{\rho_j}) \| x- y \|_{\cV_{\beta_j}} ,
    \end{split}
\end{equation*}
where in the second inequality we have used that $(a+b)^{\rho_j} \lesssim a^{\rho_j} + b^{\rho_j}.$ If $\rho_j=0$, then the term involving $z$ is just a constant and can be absorbed into the coefficient $1$ in \eqref{eq:nonaut-local-lip}. In other words,
we can choose $g_j\equiv 1\in \rL^\infty(0,T)$. Therefore, assume from now on that $\rho_j>0$. We then choose $g_j(t):=\|z(t)\|_{\cV_{\beta_j}}^{\rho_j},$
and we have to prove that $ g_j\in \rL^{\frac{2(\rho_j+1)}{\rho_j}}(0,T)$. Using \eqref{eq:A3_coefficients_bound}, it holds 
$$\theta_j := 2 \beta_j - 1 \in (0,1).
$$
Since $\cV_{\beta_j} = (\cH , \cV)_{\theta_j,2},$ by interpolation inequality it follows that 
$$\| z(t) \|_{\cV_{\beta_j}} \lesssim \| z(t) \|_{\cH}^{1- \theta_j} \| z(t) \|_{\cV}^{\theta_j}, \quad \text{ and thus } \quad \| z(t) \|_{\cV_{\beta_j}}^{2(\rho_j +1)} \lesssim \| z(t) \|_{\cH}^{2(1-\theta_j)(\rho_j +1)} \| z(t) \|_{\cV} ^{2\theta_j(\rho_j +1)}.
$$
Now, observe that by \eqref{eq:A3_coefficients_bound}, we have 
\begin{equation}
\theta_j ( \rho_j +1) = (2 \beta_j -1) ( \rho_j +1) \leq 1 .    
\label{eq: condition theta_j rho_j}
\end{equation}
Since $z \in \rC([0,T]; \cH) \cap \rL^2(0,T; \cV),$ we can apply H\"older inequality and, thanks to \eqref{eq: condition theta_j rho_j}, we obtain
\begin{equation}
    \int_0^T \| z(t) \|_{\cV_{\beta_j}}^{2(\rho_j+1)} \d t \lesssim  \| z \|_{\rC([0,T]; \cH)}^{2(1-\theta_j)(\rho_j +1)} \int_0^T \| z(t) \|_{\cV}^{2\theta_j(\rho_j +1)} \d t < \infty .
    \label{eq: check g_j integrability}
\end{equation}
Thus $g_j (t) = \| z(t) \|_{\cV_{\beta_j}}^{\rho_j} \in \rL^{\frac{2(\rho_j +1)}{\rho_j}}(0,T)$, and the assumption (2) in \autoref{lemma: new non autonomous lemma local existence} is verified.

We finally check assumption (3) in \autoref{lemma: new non autonomous lemma local existence}. Note that
$$
G_z(t,0) = F(z(t)) + \mu \rI_{\delta}( u(t) - z(t)).
$$
The second term belongs to $\rL^2(0,T; \cV^\ast)$ because $u,z \in \rL^2(0,T; \cV)$ and, thanks to \eqref{eq:Idelta-bound}, $\rI_{\delta} \in \mathcal{L}(\cV, \cV^\ast)$.  For the first term, \eqref{eq:A2_Fj_bound} with $v=0$ gives
$$
\| F(z(t)) \|_{\cV^\ast} \leq C + C\sum_{j =1}^m ( \| z(t) \|_{\cV_{\beta_j}} + \| z(t) \|_{\cV_{\beta_j}}^{\rho_j +1})
$$
By virtue of \eqref{eq: check g_j integrability},
$$
z \in \rL^2(0,T; \cV_{\beta_j}) \cap \rL^{2(\rho_j +1)}(0,T; \cV_{\beta_j}),
$$
and this in turn implies that also assumption (3) in \autoref{lemma: new non autonomous lemma local existence} is verified.

In conclusion, \autoref{lemma: new non autonomous lemma local existence} gives a unique maximal solution
\[
y\in \rL^2(0,a_{\max};\cV)\cap \rH^1(0,a_{\max};\cV^\ast)\cap \mathrm{BUC}([0,a_{\max}];\cH)
\]
on some interval $[0,a_{\max})\subseteq [0,T]$.

\emph{Step $2$: a priori estimate for $y$ and exclusion of blow-up in finite time.} Set
\[
\eta:=y+z.
\]
Testing \eqref{eq:pathwise-frozen} by $y$ and using \eqref{eq:coercive-Amudelta}, we obtain
\[
\frac12  \frac{\d}{\d t} \|y(t)\|_\cH^2 + \frac{\alpha}{2}\|y(t)\|_\cV^2
\le
\langle F(\eta(t)),y(t)\rangle_{\cV^\ast,\cV}
+
\mu \langle \rI_\delta(u(t)-z(t)),y(t)\rangle_{\cV^\ast,\cV}.
\]
\noindent 
We split the nonlinear term according to the interpolation regimes. Set
$$
    J_<: = \{j:\tfrac{1}{2}< \beta_j < \tfrac{3}{4}\},
    \qquad J_\geq : =  \{j : \tfrac{3}{4} \leq \beta_j <1\},
$$
and
$$
    \Psi_<(z) := \sum_{j\in J_<} \Psi_j(z), \qquad
    \eps_<:= \sum_{j \in J_<} \eps_j,  \qquad
    \eps_\geq := \sum_{ j \in J_\geq} \eps_j ,
$$
where $(\eps_j, \Psi_j)$ are given in \textup{\bf{(A3)}}. Since $F=\sum_{j=1}^m F_j$, we estimate separately the contributions with indices in $J_<$ and $J_\geq$.

For $j\in J_<$, applying \eqref{eq:A3shift} with $x=y(t)$ and
$y=z(t)$ gives
$$
    \langle F_j(\eta(t)) , y(t) \rangle_{\cV^\ast,\cV}
    \leq \eps_j \| y(t) \|_\cV^2 + C_j^{(0)}(1+\Psi_j(z(t)))\| y(t)\|_\cH^2
    + C_j^{(1)} \Psi_j ( z ( t ) ) .
$$
Summing over $j\in J_<$, we obtain
\begin{equation}
    \sum_{j\in J_<} \langle F_j( \eta ( t ) ) , y( t ) \rangle_{ \cV^\ast, \cV} \leq \eps_< \| y(t) \|_\cV^2
    +
    C( 1 + \Psi_< ( z ( t ) )) \| y(t) \|_\cH^2
    + C\Psi_<(z(t)).
    \label{eq:low-components-estimate}
\end{equation}
By the integrability assumption in \textup{\bf(A3)}-(ii),
$\Psi_<(z(\cdot)) \in \rL^1(0,T)$.

We now consider the components with $j\in J_\geq$. For each such $j$
$$
    \langle F_j ( \eta ( t ) ) , y( t ) \rangle_{\cV^\ast,\cV}
    = \langle F_j( \eta (t) ) , \eta ( t) \rangle_{\cV^\ast,\cV}
    - \langle F_j( \eta ( t ) ) , z(t) \rangle_{\cV^\ast,\cV}.
$$
By \eqref{eq:A3diag}, applied with $x=\eta(t)$, we have
$$
    \langle F_j( \eta(t)) , \eta(t) \rangle_{\cV^\ast,\cV}
    \leq \eps_j \| \eta(t) \|_\cV^2 + C_j^{(1)} \| \eta(t) \|_\cH^2 + C_j^{(0)} .
$$
Since $\eta(t)=y(t)+z(t)$, we also have
$$
    \| \eta(t)\|_\cV^2 \leq 2\| y( t ) \|_\cV^2 + 2 \| z ( t ) \|_\cV^2, \qquad
    \|\eta(t)\|_\cH^2 \leq 2 \| y(t) \|_\cH^2+ 2\|z(t)\|_\cH^2 .
$$
Therefore,
$$
    \langle F_j(\eta(t)),\eta(t)\rangle_{\cV^\ast,\cV}
    \leq 2\eps_j\|y(t)\|_\cV^2 + C\|y(t)\|_\cH^2 + h_{1,j}(t),
$$
where
$$
    h_{1,j}(t)
    :=
    C(1+\|z(t)\|_\cV^2+\|z(t)\|_\cH^2 ).
$$
Since $z\in \rL^2(0,T;\cV)\cap \rC([0,T];\cH)$, we have
$h_{1,j}\in \rL^1(0,T)$. We now estimate the term
$\langle F_j(\eta(t)),z(t)\rangle_{\cV^\ast,\cV}$. By
\eqref{eq:A2_Fj_bound}, applied with $v=0$, and by absorbing $F_j(0)$ into the
constant,
$$
    \|F_j(\eta(t))\|_{\cV^\ast}
    \leq
    C (1+\|\eta(t)\|_{\cV_{\beta_j}}^{\rho_j+1} ).
$$
Hence
$$
    \vert \langle F_j(\eta(t)),z(t)\rangle_{\cV^\ast,\cV}\vert 
    \leq C( 1 + \| \eta ( t) \|_{\cV_{\beta_j}}^{\rho_j+1}) \|z(t)\|_{\cV}.
$$
Set
$$
    \theta_j:=2\beta_j-1, \qquad r_j:=\rho_j+1.
$$
Since $\beta_j\ge\frac34$, we have $\theta_j\in[\tfrac{1}{2},1)$, and
\eqref{eq:A3_coefficients_bound} gives $\theta_j r_j\leq 1$. By interpolation,
$$
    \|\eta(t)\|_{\cV_{\beta_j}}^{r_j}
    \lesssim \|\eta(t)\|_{\cH}^{(1-\theta_j)r_j} \|\eta(t)\|_{\cV}^{\theta_j r_j}.
$$
Using $\theta_jr_j\le1$ and Young's inequality, for every
$\varepsilon>0$,
$$
    \|\eta(t)\|_{\cV_{\beta_j}}^{r_j}\|z(t)\|_{\cV}
    \leq
    \varepsilon\|\eta(t)\|_{\cV}^2
    +
    C_\varepsilon
    (1+\|\eta(t)\|_{\cH}^2)
    (1+\|z(t)\|_{\cV}^2).
$$
Since $\eta(t)=y(t)+z(t)$, this gives, after reducing $\varepsilon$ if needed,
$$
    \vert \langle F_j(\eta(t)),z(t)\rangle_{\cV^\ast,\cV}\vert 
    \leq \varepsilon\|y(t)\|_{\cV}^2
    + C_\varepsilon (1+\|z(t)\|_{\cV}^2) \| y(t) \|_{\cH}^2
    + h_{2,j,\varepsilon}(t),
$$
where
$$
    h_{2,j,\varepsilon}(t) := C_\varepsilon
    (1+\|z(t)\|_{\cV}^2)
    \, (1+\|z(t)\|_{\cH}^2) \in\rL^1(0,T).
$$ 
Since the number of indices is
finite, summing over $j\in J_\geq$ and reducing $\varepsilon$, if necessary,
gives
$$
    \sum_{j\in J_\geq} \vert \langle F_j(\eta) , z \rangle_{\cV^\ast,\cV}\vert 
    \leq
    \varepsilon\|y\|_\cV^2
    +  C_\varepsilon(1+\|z\|_\cV^2) \| y \|_\cH^2
    +h_{2,\varepsilon}(t),
$$
with \(h_{2,\varepsilon}\in\rL^1(0,T)\).

Combining the estimates for the low and high components, we obtain
\begin{equation*}
    \langle F(\eta(t)) , y(t)\rangle_{\cV^\ast,\cV}
    \leq (\eps_<+2\eps_\geq + \varepsilon )\| y(t)\|_\cV^2
    +a_z(t)\|y(t)\|_\cH^2
    + h_{z,\varepsilon}(t),
\end{equation*}
where \(h_{z,\varepsilon}\in\rL^1(0,T)\) and 
$
    a_z(t)
    =
    C(1+\Psi_<(z(t))+\|z(t)\|_\cV^2).
$

Let $c_{\rm emb}>0$ denote the constant involved in the embedding $\cH\hookrightarrow\cV^\ast$.
Then, by \eqref{eq:Idelta-bound},
\[
\langle \rI_\delta f,g\rangle_{\cV^\ast,\cV}
\le
(C_{\rI}\delta+c_{\rm emb})\|f\|_\cH\|g\|_\cV,
\qquad f\in\cH,\ g\in\cV.
\]
Hence, for every $\varepsilon_2>0$,
\[
\mu \langle \rI_\delta(u-z),y\rangle_{\cV^\ast,\cV}
\le
\varepsilon_2 \|y\|_\cV^2
+
C_{\varepsilon_2,\mu,\delta}\|u-z\|_\cH^2.
\]

We now choose $\varepsilon,\varepsilon_2>0$ such that
$
    \eps_<+2\eps_\ge+\varepsilon+\varepsilon_2<\frac{\alpha}{2}.$ This is possible because
$
    \eps_<+\eps_\ge=\sum_{j=1}^m\eps_j<\frac{\alpha}{4},$ and hence $\eps_<+2\eps_\ge<\frac{\alpha}{2}$. Combining the above estimates,
we arrive at
\[
    \frac12 \frac{\d}{\d t}\|y(t)\|_\cH^2 + c_0\|y(t)\|_\cV^2
    \le a_z(t)\|y(t)\|_\cH^2 + h_z(t),
\]
where $
    c_0>0$, $h_z (t)= h_{z,\varepsilon}(t)
+
C_{\varepsilon_2,\mu,\delta}\|u(t)-z(t)\|_\cH^2 \in \rL^1(0,T) $, and
\[
    a_z(t)
    =
    C\bigl(1+\Psi_<(z(t))+\|z(t)\|_\cV^2\bigr)
    \in\rL^1(0,T).
\]

Hence, by Gronwall's lemma,
\[
y\in \rL^\infty(0,t;\cH)\cap \rL^2(0,t;\cV),
\qquad \forall\, t<a_{\max}.
\]
Now $\eta=y+z$ belongs to $\rL^\infty(0,t;\cH)\cap \rL^2(0,t;\cV)$ for every $t<a_{\max}$. Hence, by $\bf(A2)$ and the standard interpolation argument, we infer that
\[
F(\eta)=F(y+z)\in \rL^2(0,t;\cV^\ast),
\qquad \forall\, t<a_{\max}.
\]
Consequently,
\[
y' = -(A+\mu \, \rI_\delta)y + F(y+z) + \mu \, \rI_\delta(u-z)
\in \rL^2(0,t;\cV^\ast),
\]
and hence
\[
\sup_{t<a_{\max}}
\|y\|_{\rL^2(0,t;\cV)\cap \rH^1(0,t;\cV^\ast)}<\infty.
\]
The blow-up alternative \eqref{eq:blow_up_non_autonomous} therefore cannot occur at finite time, and thus $a_{\max}=T$.

\emph{Step $3$: reconstruction of the stochastic solution and its measurability.} Finally, by a standard Picard-iteration argument in the class of adapted $\cH$-valued and progressively measurable $\cV$-valued processes, the pathwise solution $\hat v$ constructed above admits a version with these measurability properties. Hence the same holds for $v=\hat v+Z$.
\end{proof}

\section{Proofs of the main results}
\label{sec: proofs main results}
\noindent 
We now turn to the main objective of the paper: the synchronisation of the assimilated state with the reference trajectory. The proof of the first main theorem is
based on an It\^{o} energy estimate for the error $w=u-v$ solving \eqref{eq:model-difference-stoch}. Indeed, assumption
\textup{\bf(A4)} controls the nonlinear increment along the reference solution, while the nudging term provides the stabilizing contribution responsible for the decay. This yields exponential convergence in expectation up to the stochastic forcing term.
\begin{proof}[Proof of \autoref{thm:stoch-H}]
Applying It\^o's formula to $\|w_t\|_{\cH}^2$ and using \eqref{eq:model-difference-stoch}, we obtain
\begin{align*}
    \d \|w_t\|_{\cH}^2
    &+ 2\langle Aw_t,w_t\rangle_{\cV^\ast,\cV}\d t \\
    &=
    2\langle F(u(t))-F(v_t),w_t\rangle_{\cV^\ast,\cV}\d t
    -2\mu \langle \rI_\delta w_t,w_t\rangle_{\cV^\ast,\cV}\d t \\
    &\qquad
    + \mu^2 \|G_{\delta}(u(t))\|_{\mathrm{L}_2^0}^2\d t
    -2\mu (w_t,G_{\delta}(u(t))\d W_t^Q)_\cH .
\end{align*}
Taking expectations and using that the martingale term has zero expectation gives
\begin{align*}
    \frac{\d}{\d t}\E\|w_t\|_{\cH}^2
    + 2\E\langle Aw_t,w_t\rangle_{\cV^\ast,\cV}
    &=
    2\E\langle F(u(t))-F(v_t),w_t\rangle_{\cV^\ast,\cV}
    -2\mu \E\langle \rI_\delta w_t,w_t\rangle_{\cV^\ast,\cV} \\
    &\qquad
    + \mu^2 \|G_{\delta}(u(t))\|_{\mathrm{L}_2^0}^2 .
\end{align*}
Since $v_t=u(t)-w_t$, assumption $\bf (A4)$ yields
\[
\langle F(u(t))-F(v_t),w_t\rangle_{\cV^\ast,\cV}
=
\langle F(u(t))-F(u(t)-w_t),w_t\rangle_{\cV^\ast,\cV}
\leq \eps_1 \|w_t\|_{\cV}^2 + \kappa_u(t)\|w_t\|_{\cH}^2.
\]
Moreover, by $\bf(A1)$,
\[
\langle Aw_t,w_t\rangle_{\cV^\ast,\cV}
\geq \alpha \|w_t\|_{\cV}^2.
\]
For the nudging term, using \eqref{eq:Idelta-bound} with $f=w_t$ and $g=w_t$, we obtain
\[
-\langle \rI_\delta w_t,w_t\rangle_{\cV^\ast,\cV}
=
-\|w_t\|_{\cH}^2 + \langle w_t-\rI_\delta w_t,w_t\rangle_{\cV^\ast,\cV}
\leq
-\|w_t\|_{\cH}^2 + C_{\rI}\delta \|w_t\|_{\cH}\|w_t\|_{\cV}.
\]
Hence, by Young's inequality,
\[
-2\mu \langle \rI_\delta w_t,w_t\rangle_{\cV^\ast,\cV}
\leq
-2\mu \|w_t\|_{\cH}^2
+ \eps_2 \|w_t\|_{\cV}^2
+ C_{\eps_2}\mu^2\delta^2 \|w_t\|_{\cH}^2.
\]
Combining the above estimates and choosing $\eps_2>0$ so small that
\[
c_0:=2\alpha-2\eps_1-\eps_2>0,
\]
we arrive at
\begin{equation}\label{eq:pre-gronwall-stoch}
    \frac{\d}{\d t}\E\|w_t\|_{\cH}^2
    + c_0 \E\|w_t\|_{\cV}^2
    \leq
    \Bigl(
    -2\mu + C\mu^2\delta^2 + 2\kappa_u(t)
    \Bigr)\E\|w_t\|_{\cH}^2
    + \mu^2 \|G_{\delta}(u(t))\|_{\mathrm{L}_2^0}^2.
\end{equation}
Dropping the non-negative $\cV$-term and setting
\[
y(t):=\E\|w_t\|_{\cH}^2,
\qquad
\lambda(t):=-2\mu + C\mu^2\delta^2 + 2\kappa_u(t),
\]
we obtain
\[
y'(t)\leq \lambda(t)y(t)+\mu^2 \|G_{\delta}(u(t))\|_{\mathrm{L}_2^0}^2.
\]
By variation of constants,
\[
y(t)\leq e^{\int_0^t \lambda(r)\d r} y(0)
+ \mu^2 \int_0^t e^{\int_s^t \lambda(r)\d r} \|G_{\delta}(u(s))\|_{\mathrm{L}_2^0}^2\d s.
\]
Using \eqref{eq:A4avg}, we infer
\[
\int_s^t \lambda(r)\d r
\leq
-\Bigl( 2\mu -C\mu^2\delta^2 -2M_0\Bigr)(t-s) + 2M_1.
\]
Choose $\mu_0,\eta_0>0$ so that
\[
\gamma:=2\mu -C\mu^2\delta^2 -2M_0 >0
\]
whenever $\mu\ge \mu_0$ and $\mu\delta^2\le \eta_0$. Then
\[
e^{\int_s^t \lambda(r)\d r}\leq C e^{-\gamma (t-s)},
\]
with $C=e^{2M_1}$. Therefore
\[
y(t)
\leq
C e^{-\gamma t} y(0)
+ C\mu^2 \int_0^t e^{-\gamma (t-s)} \|G_{\delta}(u(s))\|_{\mathrm{L}_2^0}^2\d s,
\]
which is exactly \eqref{eq:main-stoch-est}.
\end{proof}
\noindent
\autoref{thm:stoch-H} is the core synchronisation estimate of the paper. The next two corollaries show how this estimate can be refined in two relevant directions: first, by quantifying the asymptotic noise floor under a uniform bound on the noise coefficient, and second, by deriving the corresponding
estimate in the weaker $\cV^\ast$-topology.
\begin{proof}[Proof of \autoref{cor:noise-floor}]
By \eqref{eq:uniform-noise-bound}, we have
\[
\int_0^t e^{-\gamma (t-s)} \|G_{\delta}(u(s))\|_{\mathrm{L}_2^0}^2\d s
\leq
\Gamma_u \int_0^t e^{-\gamma (t-s)}\d s
\leq \gamma^{-1}\Gamma_u.
\]
Inserting this into \eqref{eq:main-stoch-est} yields \eqref{eq:noise-floor-est}, and \eqref{eq:noise-floor-limsup} follows by letting $t\to\infty$. %The additive-noise case is obtained by taking $G\equiv I_\cH$.
\end{proof}

\begin{proof}[Proof of \autoref{cor:stoch-Vstar}]
Since the embedding $\cH\hookrightarrow \cV^\ast$ is continuous, there exists $C>0$ such that
\[
\|x\|_{\cV^\ast}\leq C\|x\|_{\cH},
\qquad \forall\, x\in \cH.
\]
Hence
\[
\E\|w_t\|_{\cV^\ast}^2 \leq C \E\|w_t\|_{\cH}^2,
\]
and the claim follows from \autoref{thm:stoch-H}.
\end{proof}
\noindent
We conclude with two asymptotic refinements of the synchronisation result. The
first one shows that, in the presence of a compact global attractor, the
asymptotic size of the residual error is determined only by the values of the
noise coefficient near the attractor. The second one strengthens the
mean-square convergence to almost sure convergence under an additional
integrability assumption on the stochastic forcing.
\begin{proof}[Proof of \autoref{cor: attractor noise floor}]
Since $\mathcal{A}$ is compact and $G_{\delta}$ is continuous on an open neighborhood
$\mathcal U$ of $\mathcal{A}$, the map $
a \mapsto \|G_{\delta}(a)\|_{\rL_2^0}^2
$ is continuous on $\mathcal{A}$. Hence
$$
\Gamma_{\mathcal A}:=\max_{a\in \mathcal{A}} \| G_{\delta}(a ) \|_{\rL_2^0}^2<\infty.
$$
Moreover, since $\mathcal{A}$ is a compact global attractor and $u(t)$ is a
trajectory of the deterministic semiflow generated by \eqref{eq:model}, we have
$
\mathrm{dist}_{\cH}(u(t),\mathcal A)\to 0$ as $t \to \infty.$ Therefore, for every $\varepsilon>0$, there exists $T_\varepsilon>0$ such that
$$
u(t)\in \mathcal{U} \qquad\text{and}\qquad \|G_{\delta}(u(t))\|_{\rL_2^0}^2\le \Gamma_{\mathcal{A}}+\varepsilon,
\qquad \forall\, t\ge T_\varepsilon.
$$
Now, by \eqref{eq:main-stoch-est}, for every $t\ge T_\varepsilon$,
$$
\E\|w_t\|_{\cH}^2 \leq
C e^{-\gamma t}\|w_0\|_{\cH}^2 + C\mu^2\int_0^t e^{-\gamma(t-s)}\|G_{\delta}(u(s))\|_{\rL_2^0}^2\d s.
$$
Splitting the integral at $T_\varepsilon$, we obtain
$$
\begin{aligned}
\E\|w_t\|_{\cH}^2 \leq&
C e^{-\gamma t}\|w_0\|_{\cH}^2
+ C\mu^2\int_0^{T_\varepsilon} e^{-\gamma(t-s)}\|G_{\delta}(u(s))\|_{\rL_2^0}^2\d s + C\mu^2\int_{T_\varepsilon}^{t} e^{-\gamma(t-s)}(\Gamma_{\mathcal{A}}+\varepsilon)\d s.
\end{aligned}
$$
Since $T_\varepsilon<\infty$ and $G_{\delta}(u(\cdot))\in \rL^2(0,T_\varepsilon;\rL_2^0)$,
the first two terms on the right-hand side converge to $0$ as $t\to\infty$.
For the last term,
$$
\int_{T_\varepsilon}^{t} e^{-\gamma(t-s)}\d s \leq \gamma^{-1},
$$
hence
$$
\limsup_{t\to\infty}\E\|w_t\|_{\cH}^2 \leq  \frac{C\mu^2}{\gamma}(\Gamma_{\mathcal{A}}+\varepsilon).
$$
Since $\varepsilon>0$ is arbitrary, we conclude that
$$
\limsup_{t\to\infty}\E\|w_t\|_{\cH}^2 \leq
 \frac{C\mu^2}{\gamma} \Gamma_{\mathcal{A}}.
$$
\end{proof}

\begin{proof}[Proof of \autoref{thm:as-convergence-multiplicative}]
The idea of the proof is the following. In the first part, we show that the discrete-time second moments and the expected suprema on unit intervals are summable; in the second part, we use Markov's inequality and Borel-Cantelli, to obtain the almost sure convergence.

 \textit{Step 1.} Set
$$
X(t):=\|w_t\|_{\cH}^2, \qquad
h(t) : = \mu^2\| G_{\delta}( u(t))\|_{ \rL_2^0}^2.
$$
By assumption \eqref{eq:extra-integrability-Gutilde} we have
\begin{equation}
h \in \rL^1(0,\infty).
\label{eq:h-L1}
\end{equation}
From \autoref{thm:stoch-H}, with $
y(t):=\mathbb E X(t)=\mathbb E\|w_t\|_{\cH}^2, $
we know that
\begin{equation}
y(t) \leq K e^{-\gamma t}y(0)+K\int_0^t e^{-\gamma(t-s)}h(s)\d s,
\qquad t\ge0.
\label{eq:y-est-for-as}
\end{equation}
Evaluating at integer times and summing over $n\ge0$, we get
\begin{equation}
    \begin{split}
        \sum_{n=0}^ \infty y(n)
& \leq K y(0) \sum_{n=0}^\infty e^{-\gamma n}  +K \sum_{n=0}^\infty \int_0^n e^{-\gamma(n-s)}h(s)\d s  = K y(0) \sum_{n=0}^\infty e^{-\gamma n} +K\int_0^\infty \left(\sum_{n \geq \lceil s\rceil} e^{-\gamma(n-s)}\right) h(s)\d s.
    \end{split}
    \label{eq:sum-y-n-step}
\end{equation}
The geometric sum in parentheses is bounded uniformly in $s$ (since it is a geometric series with ratio $e^{-\gamma} <1$). Hence by \eqref{eq:h-L1},
\begin{equation}
\sum_{n=0}^\infty \mathbb{E} \|w_n\|_{\cH}^2 <\infty.
\label{eq:sum-discrete-L2}
\end{equation}

 \textit{Step 2.} Next, recall from the proof of \autoref{thm:stoch-H} that
\begin{equation}
\d X(t)+c_0 \|w_t\|_{\cV}^2\d t
\leq a(t)X(t)\d t + C h(t)\d t + \d M_t,
\label{eq:ito-as-proof}
\end{equation}
where
$$
a(t):=2\kappa _u(t)+C\mu^2\delta^2-2\mu, \qquad 
M_t:=-\mu\int_0^t \langle w(s),G_{\delta}(  u(s)) \d W_s^Q\rangle_{\cH}.
$$
Fix $n\in\mathbb N$. For $t\in[n,n+1]$, integrating \eqref{eq:ito-as-proof} from $n$ to $t$ and dropping the non-negative second term on the left-hand side, we obtain
$$
X(t) \leq X(n) + C \int_{n}^t a(s) X (s) \d s + C \int_n^t h(s) \d s + M_t - M_n.
$$
For $t \in [n,n+1]$, set
$$
A_n(t) := X(n) + C \int_{n}^t h(s) \d s + \sup_{s \in [n,t]} \vert M_s - M_n \vert.
$$
Then $t \mapsto A_n(t)$ is non-decreasing, and since $a(s) \leq a_+(s):= \max \{ 0, a(s)\}$
$$
X(t) \leq A_n (t) + C \int_n^t a_{+} (s) X(s) \d s, \qquad t \in [n,n+1].
$$
Applying Gronwall Lemma, we obtain
$$
X(t) \leq A_n (t) \exp \left( C \int_{n}^t a_+ (s) \d s \right), \qquad t \in [n,n+1].
$$
Since
$$
a_+(s) \leq 2 \kappa_u(s)  + C \mu^2 \delta^2 + 2\mu,
$$
by \eqref{eq:A4avg} in assumption \textup{(A4)} we have
$$
\int_n^{n+1} a_{+} (s) \d s \leq 2 \int_n^{n+1}\kappa_u (s) \d s  + C \mu^2 \delta^2 + 2\mu \leq 2M_0  + 2M_1  + C \mu^2 \delta^2 + 2\mu ,
$$
where all the constants on the right-hand side are independent of $n$. Therefore, there exists a constant $C_1>0$ independent of $n$, such that taking the supremum over $t \in [n,n+1]$ leads to
\begin{equation}
\sup_{t\in[n,n+1]}X(t)
\leq C_1\left(
X(n)+\int_n^{n+1} h(s)\d s
+\sup_{t\in[n,n+1]} \vert M_t-M_n \vert
\right)
\qquad \mathbb{P}\text{-a.s.}
\label{eq:local-sup-before-bdg}
\end{equation}
Set $S_n:=\sup_{t\in[n,n+1]}X(t).$
We would like to apply Burkholder-Davis-Gundy (BDG) inequality. To do this, we first need to compute the quadratic
variation of the real-valued martingale $M$. Observe that
$$
\langle w(s),G_{\delta}( {u}(s))\d W_s^Q\rangle_{\cH}
= \langle G_{\delta}(  u(s))^* w(s),\d W_s^Q\rangle_{\cH_0},
$$
where $G_{\delta}(  u(s))^* w(s)\in \cH_0$, and in the previous identity we have used that since $G_{\delta}(  u(s))\in \rL_2(\cH_0,\cH)$, then its adjoint
$G_{\delta}(  u(s))^\ast$ maps $\cH$ to $\cH_0$. Hence
$$
M_t = -\mu \int_0^t \langle G_{\delta}( {u}(s))^* w(s),\d W_s^Q\rangle_{\cH_0}
$$
is a continuous real-valued martingale with quadratic variation
$$
\langle M\rangle_t =
\mu^2\int_0^t \|G_{\delta}( {u} (s) )^* w(s)\|_{\cH_0}^2\d s.
$$
Therefore, for $t\in[n,n+1]$,
$$
\langle M-M_n \rangle_t
= \mu^2\int_n^t \|G_{\delta}(  u(s))^* w(s)\|_{\cH_0}^2\d s.
$$
Applying the BDG inequality to the martingale
$(M_t-M_n)_{t\in[n,n+1]}$, we obtain
$$
\mathbb{E}\sup_{t\in[n,n+1]} \vert M_t-M_n \vert\leq C
\mathbb{E}\left(\langle M-M_n\rangle_{n+1}\right)^{1/2} = C\mu
\mathbb{E}\left(
\int_n^{n+1}\|G_{\delta}(  u(s))^* w(s)\|_{\cH_0}^2\d s
\right)^{1/2}.
$$
Taking expectations in \eqref{eq:local-sup-before-bdg} and applying the BDG inequality,
\begin{equation}
\mathbb{E} \sup_{t\in[n,n+1]}  \vert M_t-M_n \vert  \leq C\mu
\mathbb{E}\left(
\int_n^{n+1} \| G_{\delta} ( {u}(s))^\ast  w(s) \|_{\cH_0}^2 \d s
\right)^{1/2} \leq C \mu
\mathbb{E}\left(
S_n^{1/2}
\left(\int_n^{n+1}\|G_{\delta}( {u}(s)) \|_{\rL_2^0}^2\d s\right)^{1/2}
\right).
\label{eq:bdg-step}
\end{equation}
Using Young's inequality, we obtain
\begin{equation}
\mathbb{E}\sup_{t\in[n,n+1]} \vert M_t-M_n \vert 
\leq \frac{1}{2C_1} \mathbb{E} S_n
+ C \mu^2 \int_n^{n+1}\| G_{\delta}( {u}(s) )\|_{ \rL_2^0}^2\d s.
\label{eq:young-bdg-step}
\end{equation}
Combining \eqref{eq:local-sup-before-bdg} and \eqref{eq:young-bdg-step}, and absorbing the term
$\frac{1}{2} \mathbb{E} \left[S_n \right]$ into the left-hand side, gives
\begin{equation}
\mathbb{E} \left[S_n \right] \leq C_2\left ( \mathbb{E}\|w_n\|_{\cH}^2 +\int_n^{n+1} h(s)\d s \right ),  \qquad n \geq 0,
\label{eq:Sn-estimate}
\end{equation}
for some constant $C_2>0$ independent of $n$. Summing \eqref{eq:Sn-estimate} over $n\ge0$, and using \eqref{eq:sum-discrete-L2} together with \eqref{eq:h-L1}, we conclude that $
\sum_{n=0}^\infty \mathbb{E} \left[ S_n \right] < \infty.$

 \textit{Step 3.} Finally, for every $\varepsilon>0$, Markov's inequality gives
$$
\sum_{n=0}^\infty
\mathbb{P} \left(S_n>\varepsilon^2\right) \leq \frac{1}{\varepsilon^2} \sum_{n=0}^\infty \mathbb{E} \left[S_n \right]
<\infty.
$$
Hence, by the Borel-Cantelli lemma,
$$
S_n=\sup_{t\in[n,n+1]}\|w_t\|_{\cH}^2 \to 0 \qquad \mathbb{P}\text{-a.s.}
$$
This implies $w_t\to0$ in $\cH$, almost surely, as $t\to\infty$. Lastly, to prove the uniform convergence of the tails in \eqref{eq: uniform convergence of the tail} observe that
$$
\sup_{t \geq N} \| w_t \|_{\cH} = \sup_{n \geq N} \sup_{t \in [n, n+1]} \| w_t \|_{\cH} = \sup_{n \geq N} S_n^{1/2}.
$$
Thus \eqref{eq: uniform convergence of the tail} follows from the fact that $S_n \to 0 .$
\end{proof}

\appendix

\section{A local existence result for non-autonomous semilinear equations}
\label{sec:appendix}
\noindent 
This appendix contains a non-autonomous local well-posedness result used in the proof of \autoref{prop:DAmult-wp}. The result is a variant of the autonomous local theory recalled in \autoref{lem: local}, adapted to semilinear equations allowing the local Lipschitz constants to depend on time. We include the proof for completeness, since the random PDE \eqref{eq:pathwise-random-pde} falls exactly in this framework.

\begin{lem}\label{lemma: new non autonomous lemma local existence}
Let $T>0$. Assume $\widetilde{A} \in \mathcal{L}(\cV, \cV^\ast)$ and that $\widetilde{A}$ enjoys $\rL^2$-maximal regularity on $(0,T).$ Let $ \widetilde  G:(0,T) \times \cV \to \cV^\ast $ and assume the following.
\begin{enumerate}
    \item[(1)] Fix $x \in \cV$. The map $t \mapsto \widetilde  G(t,x)$ is strongly-measurable as a map with values in $\cV^*$.
    \item[(2)] For every $j=1,\ldots,m$, let $\beta_j\in(\frac12,1)$,
$\rho_j\geq0$, and assume that $(\beta_j,\rho_j)_j$ satisfy
\eqref{eq:A3_coefficients_bound}, i.e. $(2\beta_j -1) ( \rho_j +1) \leq 1$. Set
$$
r_j:=
\begin{cases}
\dfrac{2(\rho_j+1)}{\rho_j}, & \rho_j>0,\\
\infty, & \rho_j=0.
\end{cases}
$$
There exists a measurable function $g_j\in \rL^{r_j}(0,T)$ such that, for a.e.
$t\in(0,T)$ and every $x,y\in\cV$
    \begin{equation}
 \|  \widetilde  G(t,x) -  \widetilde G(t,y) \|_{\cV^\ast } \leq C \sum_{j = 1}^m ( 1+ \vert g_j (t)  \vert + \| x \|_{\cV_{\beta_j}}^{\rho_j} + \| y \|_{\cV_{\beta_j}}^{\rho_j} ) \| x-y \|_{\cV_{\beta_j}}.
        \label{eq:nonaut-local-lip}
    \end{equation}
    \item[(3)] $ \widetilde  G(\cdot, 0) \in \rL^2(0,T; \cV^\ast )$.
    \end{enumerate}
    Then, for any $u_0 \in \cH$, there exists $a = a(u_0) \in (0,T]$ such that
    \begin{equation}
    \left \lbrace 
        \begin{aligned}
            u'(t) + \widetilde{A} u(t) &= \widetilde  G(t,u(t)), \qquad t \in (0,a),\\
            u(0) & = u_0,
        \end{aligned}
        \right. 
        \label{eq: nonautonomous problem tilde a}
    \end{equation}
    admits a unique solution 
    $$
    u \in \rL^2(0,a; \cV) \cap \rH^1(0, a; \cV^\ast) \cap \mathrm{BUC}([0,a]; \cH).
    $$
    Moreover, the solution can be extended to a maximal interval of existence $[0, a_{\max}(u_0)) \subseteq [0,T]$. If $a_{\max}(u_0) < T$, then
    \begin{equation}
    \lim \limits_{t \to a_{\max}(u_0)} \| u \|_{\rL^2(0,t; \cV) \cap \rH^1(0,t; \cV^\ast )} = + \infty.
        \label{eq:blow_up_non_autonomous}
    \end{equation}
\end{lem}
\begin{proof}
For $a\in (0,T]$ set $\mathbb{E}_a:= \rL^2(0,a;\cV) \cap\rH^1(0,a;\cV^\ast)$ and $ \mathbb{F}_a:=\rL^2(0,a;\cV^\ast)$. On $\mathbb{E}_a$, we consider the maximal regularity norm $
\| u \|_{\mathbb{E}_a} := \| u \|_{\rL^2(0,a; \cV)} + \| u' \|_{\rL^2(0,a; \cV^\ast)}.$ Since $\widetilde{A}$ enjoys $\rL^2$-maximal regularity on finite intervals, for
every $a\in (0,T]$, $f\in \mathbb F_a$, and $u_0\in \cH$, the linear problem
\begin{equation*}
    \left \lbrace
    \begin{aligned}
        u'(t)+\widetilde{A}u(t) &= f(t), \qquad t\in (0,a),\\
        u(0)&=u_0,
    \end{aligned}
    \right.
\end{equation*}
admits a unique solution $u \in \mathbb{E}_a$. Moreover, there exists a constant $M_T>0$,
independent of $a\in(0,T]$, such that
\begin{equation}\label{eq:MR-nonaut-proof}
\|u\|_{\mathbb{E}_a} \leq M_T (\| u_0 \|_{\cH}+ \| f \|_{\mathbb{F}_a} ).
\end{equation}

For $j=1,\dots,m$, define
$$
\theta_j := 2\beta_j-1 \in (0,1), \qquad q_j:=2(\rho_j+1), \qquad
r_j : =  \begin{cases}
    \frac{2(\rho_j+1)}{\rho_j}, \; &\text{ if } \rho_j >0,\\ +\infty, \; &\text{ if } \rho_j = 0,
\end{cases}\qquad \sigma_j:= \frac{1- \theta_j (\rho_j +1)}{2(\rho_j +1)} \geq 0.
$$
Throughout the proof, when $r_j=\infty$ we use the conventions $\frac{1}{r_j}=0$ and $a^{\frac{1}{r_j}}=1$. Since $\cV_{\beta_j}=(\cH,\cV)_{\theta_j,2}$, interpolation inequality gives
$$
\|u(t)\|_{ \cV_{ \beta_j}}\leq C\|u(t)\|_{\cH}^{1-\theta_j}\|u(t)\|_{\cV}^{ \theta_j } \qquad \text{for a.e. } t \in (0,a).
$$
Using that \eqref{eq:A3_coefficients_bound} is equivalent to $\theta_j(\rho_j+1)\leq 1$, we obtain by interpolation and H\"older inequality
$$
\|u\|_{\rL^{q_j}(0,a;\cV_{\beta_j})}^{q_j}\leq C \|u\|_{\rC([0,a];\cH)}^{(1-\theta_j)q_j}
\int_0^a \|u(t)\|_\cV^{\theta_j q_j} \d t \leq C \|u\|_{\rC([0,a];\cH)}^{ (1- \theta_j) q_j}
a^{1- \frac{\theta_j q_j}{2}}
\|u\|_{\rL^2(0,a;\cV)}^{ \theta_j q_j}.
$$
Since $q_j = 2(\rho_j +1)$, we have $1- \frac{\theta_j q_j }{2} = 1- \theta_j (\rho_j +1)$. Considering the $q_j$th square root of the previous chain of inequalities, we obtain
$$
\| u \|_{\rL^{q_j} (0,a ; \cV_{\beta_j} )} \leq C a^{\sigma_j} \| u \|_{ \rC ([0,a]; \cH )}^{1- \theta_j} \| u \|_{\rL^2 (0,a; \cV)}^{\theta_j}.
$$
Since $\mathbb{E}_a \hookrightarrow \rC([0,a]; \cH)$, we conclude that
\begin{equation}
\|u\|_{\rL^{q_j}(0,a;\cV_{\beta_j})}
\le
C_T a^{\sigma_j}(\|u(0) \|_{\cH} + \| u \|_{\mathbb{E}_a})^{1-\theta_j}\|u\|_{\mathbb E_a}^{\theta_j}.
    \label{eq:embedding-Ej_nonzero_ic}
\end{equation}
In particular, if $u \in \mathbb{E}_a$ and $u(0) = 0$,
\begin{equation}\label{eq:embedding-Ej}
\|u\|_{\rL^{q_j}(0,a;\cV_{\beta_j})}
\le
C_T a^{\sigma_j}\|u\|_{\mathbb E_a}.
\end{equation}

\emph{Step 1: construction of the fixed point map.}
Let $z \in \mathbb{E}_T$ be the unique solution of
$$
z'(t)+\widetilde{A} z(t)= \widetilde  G(t,0), \qquad t\in (0,T),\qquad z(0) = u_0.
$$
For $a\in(0,T]$, set $z_a:=z|_{(0,a)}$. By virtue of \eqref{eq:MR-nonaut-proof} and assumption $(3)$, we have
$$
\| z \|_{\mathbb{E}_T}  \leq M_T ( \| u_0 \|_{\cH} + \| \widetilde  G ( \cdot, 0)\|_{\mathbb{F}_T}).
$$
Moreover, for any $j=1,...,m$
\begin{equation}\label{eq:za-small}
\| z_a \|_{\rL^{q_j} (0,a;\cV_{\beta_j})} = \| z \|_{\rL^{q_j} (0,a;\cV_{\beta_j})} \to 0 \qquad \text{as } a \to 0.
\end{equation}

Let $
\mathbb{E}_{a,0} := \{ w \in \mathbb{E}_a : w(0)=0 \}.$ Fix $R>0$ and define $
\mathbb{B}_{a,R} := \{ w \in \mathbb{E}_{a,0} : \| w \|_{\mathbb{E}_a} \leq R \}.$ 
Define
$$
\Psi_a: \mathbb{B}_{a,R} \to \mathbb{E}_{a,0}, \qquad \Psi_a (w):= v,
$$
where $v$ is the unique solution of
\begin{equation*}
\left \lbrace 
    \begin{aligned}
        v'(t)+ \widetilde{A} v(t) &= \widetilde  G (t, z_a(t) + w(t) )-  \widetilde  G(t,0),
\qquad t \in (0,a),\\
 v(0)&=0.
    \end{aligned}
    \right.
\end{equation*}
Note that, by the maximal regularity theory recalled above, $\Psi_a$ is well defined, if we are able to show that the right-hand side belongs to $\mathbb{F}_a$. To show this, set $ \phi_j ( t ) : = \| z_a(t) + w(t) \|_{\cV_{\beta_j}}.$ By \eqref{eq:embedding-Ej_nonzero_ic}, $\phi_j\in \rL^{q_j}(0,a)$. Moreover, by \eqref{eq:nonaut-local-lip} with $y=0$
$$
\|  \widetilde  G( t, z_a + w)- \widetilde  G( t, 0)\|_{\cV^\ast} \leq C \sum_{j=1}^m
\left[(1+|g_j(t)|)\phi_j(t)+\phi_j(t)^{\rho_j+1}\right].
$$
Since $\frac{1}{q_j}+ \frac{1}{r_j} = \frac{1}{2},$ H\"older's inequality gives
$$
\| ( 1 + \vert g_j \vert) \phi_j\|_{\rL^2(0,a )} \leq
( a^{\frac{1}{r_j}} + \| g_j \|_{\rL^{r_j}(0,a )}) \| \phi_j \|_{\rL^{q_j}(0,a)}.
$$
Furthermore, since $q_j=2(\rho_j+1)$, it holds $ \| \phi_j^{\rho_j+1} \|_{ \rL^2(0,a )} = \| \phi_j \|_{\rL^{q_j}(0,a )}^{\rho_j+1}.$
Hence each term on the right-hand side belongs to $\rL^2(0,a)$.

\emph{Step $2$: $\Psi_a$ maps $\mathbb{B}_{a,R}$ in itself.} Applying the maximal regularity estimate in \eqref{eq:MR-nonaut-proof} to $\Psi_a(w),$ and the estimates in Step $1$, we have
\begin{equation}
    \begin{split}
        \| \Psi_a (w) \|_{\mathbb{E}_a} \leq M_T \|  \widetilde  G( \cdot, z_a +w) - \widetilde   G(\cdot, 0) \|_{\mathbb{F}_a} \leq C_T \sum_{j =1}^m &\Big[(a^{\frac{1}{r_j}}+ \| g_j \|_{\rL^{r_j}(0,a)}) \| z_a + w \|_{\rL^{q_j}(0,a; \cV_{\beta_j})} \\
        &\quad + \| z_a + w \|_{\rL^{q_j}(0,a; \cV_{\beta_j})}^{\rho_j +1}\Big].
    \end{split}
    \label{eq: bound norm Psi a w}
\end{equation}
Using the triangle inequality, \eqref{eq:embedding-Ej}, and the fact that $w \in \mathbb{B}_{a,R}$, we have
\begin{equation*}
    \begin{split}
        \| z_a + w \|_{\rL^{q_j}(0, a; \cV_{\beta_j})} &\leq \| z_a \|_{\rL^{q_j}(0,a; \cV_{\beta_j})} + \| w \|_{\rL^{q_j}(0,a; \cV_{\beta_j})} \leq \| z_a \|_{\rL^{q_j}(0,a; \cV_{\beta_j})} + C_T a^{\sigma_j} \| w \|_{\mathbb{E}_a} \\
        &\leq \| z_a \|_{\rL^{q_j}(0,a; \cV_{\beta_j})} + C_T a^{\sigma_j} R=: K_j(a,R).
    \end{split}
\end{equation*}
Hence, \eqref{eq: bound norm Psi a w} can be rewritten as
$$
\| \Psi_a (w) \|_{\mathbb{E}_a} \leq C_T \sum_{j = 1}^m \left[ (a^{\frac{1}{r_j} } + \| g_j \|_{\rL^{r_j}(0,a)} )K_j (a,R) + K_j (a,R)^{\rho_j +1}  \right].
$$
It remains to be shown that the right-hand side can be made smaller than $R$ for a suitable choice of $a,R$.

If $\sigma_j >0,$ then $K_j(a,R) \to 0$ as $a \to 0$ thanks to \eqref{eq:za-small}.

If $\sigma_j = 0,$  then $\theta_j (\rho_j +1) = 1.$ Since $\theta_j \in (0,1),$ this implies $\rho_j >0$ and thus $r_j <\infty .$ Thus
$$
a^{\frac{1}{r_j}} + \| g_j \|_{\rL^{r_j}(0,a )} \to 0 \qquad \text{ as } a \to 0,
$$
Furthermore, $K_j(a,R) \to C_T R$ as $a \to 0$. These imply two facts. First, the terms $(a^{\frac{1}{r_j} } + \| g_j \|_{\rL^{r_j}(0,a)} )K_j (a,R)$ can be made arbitrarily small. Second, for the remaining terms we have $K_j(a,R)^{\rho_j+1} \sim R^{\rho_j +1}$. In conclusion, we can find $R,a$ such that $\|\Psi_a (w) \|_{\mathbb{E}_a} \leq R$.

\emph{Step $3$: $\Psi_a$ is a contraction.} Let $v_i = \Psi_a (w_i)$, for $i = 1,2$ and $w_i \in \mathbb{B}_{a,R}$. Then $v:=  v_1 -v_2 \in \mathbb{E}_{a}$ satisfies
\begin{equation*}
\left \lbrace 
    \begin{split}
        v'(t) + \widetilde{A} v(t) & = \widetilde  G(t, z_a(t) + w_1(t)) -  \widetilde  G(t, z_a(t) + w_2(t)), \qquad t \in (0,a), \\
        v(0) & = 0.
    \end{split}
    \right. 
\end{equation*}
By maximal regularity
$$
\| \Psi_a (w_1) - \Psi_a (w_2) \|_{\mathbb{E}_a} \leq M_T \| \widetilde  G (\cdot, z_a + w_1) - \widetilde  G(\cdot, z_a + w_2) \|_{\mathbb{F}_a}.
$$
We now estimate the norm on the right-hand side. By \eqref{eq:nonaut-local-lip}, we have
$$
\| \widetilde  G( \cdot, z_a + w_1) - \widetilde  G(\cdot, z_a + w_2) \|_{\cV^\ast } \leq C \sum_{j = 1}^m (1+ \vert g_j \vert+ \| z_a + w_1 \|_{\cV_{\beta_j}}^{\rho_j} + \| z_a + w_2 \|_{\cV_{\beta_j}}^{\rho_j} )\| w_1- w_2 \|_{\cV_{\beta_j}}.
$$
Considering the $\rL^2(0,a)$ norm, we estimate the three types of terms separately. First, using H\"older inequality with $\frac{1}{r_j}+ \frac{1}{q_j} = \frac{1}{2},$ we have
$$
\| w_1 - w_2 \|_{\rL^2(0,a; \cV_{\beta_j})} \leq a^{\frac{1}{r_j}} \| w_1 - w_2 \|_{\rL^{q_j}(0,a; \cV_{\beta_j})},
$$
where, if $\rho_j  = 0$, the previous coefficients are $q_j = 2$, $r_j = \infty $, and $a^{\frac{1}{r_j}} = 1$. Second, again by H\"older inequality with $\frac{1}{r_j}+ \frac{1}{q_j} = \frac{1}{2}$, it holds
$$
\| |g_j| \|w_1-w_2\|_{\cV_{\beta_j}}\|_{\rL^2(0,a)}
\leq \| g_j \|_{\rL^{r_j}(0,a)} \|w_1-w_2\|_{\rL^{q_j}(0,a;\cV_{\beta_j})}.
$$
For the non-linear terms, if $\rho_j >0$, then $\rho_j r_j = q_j$. Again thanks to H\"older inequality with the same exponents, we get to
\begin{equation*}
    \begin{split}
        \left\| \, \| z_a + w_i \|_{\cV_{\beta_j}}^{\rho_j} \| w_1- w_2 \|_{\cV_{\beta_j}}\, \right\|_{\rL^2(0,a)} \leq \| z_a + w_i \|_{\rL^{q_j}(0,a; \cV_{\beta_j})}^{\rho_j} \| w_1-w_2 \|_{\rL^{q_j}(0,a; \cV_{\beta_j})}.
    \end{split}
\end{equation*}
If $\rho_j = 0$, then the previous term simplifies to just
$
\| w_1 - w_2 \|_{\rL^2(0,a; \cV_{\beta_j})}  = \| w_1 - w_2 \|_{\rL^{q_j}(0,a; \cV_{\beta_j})}.$ In conclusion, with the convention $K_j(a,R) ^0 = 1$, we arrive at
\begin{equation*}
    \begin{split}
         \| \widetilde  G( \cdot, z_a + w_1) - \widetilde  G(\cdot, z_a + w_2) \|_{\mathbb{F}_a } \leq C \sum_{j =1}^m &\Big( a^{\frac{1}{r_j}} + \| g_j \|_{\rL^{r_j}(0,a)} + \| z_a +w_1 \|_{\rL^{q_j}(0,a; \cV_{\beta_j})}^{\rho_j}  \\
        & \quad + \| z_a + w_2 \|_{\rL^{q_j}(0,a; \cV_{\beta_j})}^{\rho_j} \Big) \| w_1 - w_2 \|_{\rL^{q_j}(0,a; \cV_{\beta_j})}.
    \end{split}
\end{equation*}
Using \eqref{eq:embedding-Ej} to bound the norms $\| \cdot \|_{\rL^{q_j}(0,a; \cV_{\beta_j})}$ and using that $\| w_i \|_{\mathbb{E}_a} \leq R,$ we obtain
$$
\|  \Psi_a (w_1 ) - \Psi_a(w_2) \|_{\mathbb{E}_a} \leq C_T\left[ \sum_{j = 1}^ma^{\sigma_j} \left( a^{\frac{1}{r_j}} + \| g_j \|_{\rL^{r_j} (0, a)} + 2K_j (a,R)^{\rho_j} \right) \right] \| w_1 - w_2 \|_{\mathbb{E}_a}.
$$
We now choose $R>0$ and $a>0$ so that the previous Lipschitz constant is strictly smaller than $1$. Indeed, if $\sigma_j>0$, then the corresponding contribution tends to zero as $a\to0$. This also covers the endpoint
$\rho_j=0$, because then $\sigma_j=\frac{1-\theta_j}{2}>0$, while
$r_j=\infty$ and $K_j(a,R)^0=1$. If instead $\sigma_j=0$, then $\theta_j(\rho_j+1)=1$, and therefore $\rho_j>0$ and $r_j<\infty$. Hence
$$
a^{\frac{1}{r_j}}+\|g_j\|_{\rL^{r_j}(0,a)}\to0
\qquad \text{as } a\to0,
$$
and $K_j(a,R)^{\rho_j}\to (C_T R)^{\rho_j}$. Choosing first $R>0$ small enough and then $a>0$ small enough, the full Lipschitz constant is strictly smaller than $1$. 

By the contraction mapping principle, there exists a unique fixed point of $w \in \mathbb{B}_{a,R}$ for $\Psi_a.$ Then, $u := z_a +w \in \mathbb{E}_a$ solves the original problem \eqref{eq: nonautonomous problem tilde a} on $(0,a)$.

\emph{Step $4$: uniqueness in $\mathbb{E}_a$ and maximal interval of existence.} Regarding the uniqueness in $\mathbb{E}_a,$ let $u_i \in \mathbb{E}_a$, for $i =1,2$, be two solutions of \eqref{eq: nonautonomous problem tilde a} with the same initial condition $u_0$. Hence, for $s,t\in (0,a]$, with $s<t$, their difference $v:= u_1 - u_2$ solves
\begin{equation*}
\left \lbrace 
    \begin{split}
        v'(r) + \widetilde{A} v (r) &= \widetilde  G(r, u_1(r)) - \widetilde  G(r, u_2(r)), \qquad r \in (s,t), \\
        v(s)& = u_1(s) -u_2(s).
    \end{split}
    \right.
\end{equation*}
Repeating the contraction estimate of Step $3$ on the arbitrary time interval $(s,t) \subset (0,a)$, with $u_i$ replacing $z_a + w_i$, and assuming that $u_1(s) = u_2(s)$, it is possible to check that
$$
\| u_1 - u_2 \|_{\mathbb{E}_{s,t}} \leq C_T  L(s,t) \| u_1 - u_2 \|_{\mathbb{E}_{s,t}},
$$
where
$$
L(s,t):= \sum_{j = 1}^m (t-s)^{\sigma_j} \left[ (t-s)^{\frac{1}{r_j}} + \| g_j \|_{\rL^{r_j}(s,t)} + \| u_1\|_{\rL^{q_j}(s,t;V_{\beta_j})}^{\rho_j} +\| u_2\|_{\rL^{q_j}(s,t;V_{\beta_j})}^{\rho_j} \right],
$$
and $ \mathbb{E}_{s,t}:= \rL^2(s,t; \cV) \cap \rH^1(s,t; \cV^\ast).$ 

We now check that $L(s,t)\to 0$ as $t \to s$. Indeed, if $\sigma_j >0$, then the factor $(t-s)^{\sigma_j}$ gives smallness of $L(s,t)$. If $\sigma_j =0$, then, as in Step $3$, we have $\rho_j >0$ and $r_j <\infty.$ So, the smallness follows from the absolute continuity of the norms in $\rL^{r_j}$ and $\rL^{q_j}$.
Hence, for any $s < a$, such that $u_1 = u_2 \in \cH$ on $[0, s]$, there exists $t>s$ such that $C_T L(s,t) <1$, and therefore $u_1 = u_2$ also on $[s,t].$ 
Let $\tilde a := \inf \{ t \in [0,a] \, : \, u_1(t) \neq u_2(t) \}.$ Since $u_1, u_2$ are continuous on $[0,a]$ with values in $\cH$ we must have $u_1(\tilde{a}) = u_2(\tilde{a}) \in \cH$.
The previous argument shows that $u_1 = u_2$ also on $[\tilde{a},\tilde{a} + h]$ for sufficiently small $h > 0$ which is a contradiction to the definition of $\tilde{a}$, and we have proved uniqueness on the whole interval $[0,a].$

Lastly, the existence of a maximal solution follows from a classical argument based on considering $a_{\max} (u_0):= \sup \{ \tau \in (0,T] : \text{ a solution exists on } (0,\tau) \}$ and then gluing together (thanks to uniqueness) the local in time solutions constructed with the fixed point argument.
\end{proof}

{\bf Acknowledgements. }{\small Gianmarco Del Sarto, Matthias Hieber and Tarek Z\"{o}chling acknowledge the support from the DFG project FOR~5528. The research of Filippo Palma is carried out under the auspices of GNFM-INdAM. Part of the research was carried out during a visit of Jochen Br\"{o}cker to Darmstadt, and JB acknowledges generous support and hospitality during this visit.}

{\bf Data availability statement. }{\small Data sharing not applicable to this article as no datasets were
generated or analyzed during the current study.}

{\bf Conflict of interest. }{\small  On behalf of all the authors, the corresponding author states
that there is no Conflict of interest.}

{\bf Author contributions. }{\small All authors contributed to the conception of the work. They also were involved in drafting or revising the article critically for intellectual content. All approved the final version.}

%\bibliographystyle{plain}
%\bibliography{biblio}

\end{document}